\documentclass [9pt,a4paper]{article}
\usepackage[T1]{fontenc}
\usepackage[ansinew]{inputenc}
\usepackage{amssymb}
\usepackage{amsmath}
\usepackage{graphicx}
\usepackage{amsthm}
\usepackage{cite}
\usepackage{color}
\usepackage{lmodern,soul}
\usepackage{amsfonts}
\usepackage{mathrsfs}

\newtheorem{theorem}{\textbf{Theorem}}[section]
\newtheorem{lemma}{\textbf{Lemma}}[section]
\newtheorem{proposition}{\textbf{Proposition}}[section]
\newtheorem{corollary}{\textbf{Corollary}}[section]
\newtheorem{remark}{\textbf{Remark}}[section]

\numberwithin{equation}{section}

\allowdisplaybreaks[4]

\def\non{\nonumber}

\newcommand\defeq{\stackrel{\scriptscriptstyle \text{def}}=}
\newcommand{\sS}{\mathbb{S}^2}
\newcommand\eps{{\varepsilon}}
\newcommand{\RR}{\mathbb{R}}
\newcommand{\A}{\mathcal{A}}
\newcommand{\E}{\mathcal{E}}
\newcommand{\F}{\mathcal{F}}

\newcommand{\n}{\mathbf{n}}

\newcommand{\IO}{\mathrm{I}_0}

\newcommand{\ud}{\mathrm{d}}
\newcommand{\tr}{\operatorname{tr}}
\newcommand{\Q}{\mathbb{Q}}

\begin{document}

\title{Blowup rate estimates of the Ball-Majumdar potential and its gradient in the Landau-de Gennes theory}

\author{
  Xin Yang Lu\footnote{Department of Mathematical Sciences, Lakehead University, Thunder Bay, Canada
  	and
  	Department of Mathematics and Statistics, McGill University, Montr\'eal, Canada.
  \texttt{xlu8@lakeheadu.ca}} \,
  \and
  Xiang Xu\footnote{
  Department of Mathematics and Statistics,
  Old Dominion University, Norfolk, VA 23529, USA.
  \texttt{x2xu@odu.edu}}
  \and
  Wujun Zhang\footnote{
  Department of Mathematics,
  Rutgers University, New Brunswick, NJ 08901, USA.
  \texttt{wujun@math.rutgers.edu}}
 }

\date{}
\maketitle

\begin{abstract}
In this paper we revisit a singular bulk potential in the Landau-de Gennes free energy that
describes nematic liquid crystal configurations in the framework of the $Q$-tensor order parameter. This singular potential, called Ball-Majumdar potential, is introduced in \cite{BM10}.
It is considered as a natural enforcement of a physical constraint on
the eigenvalues of symmetric, traceless $Q$-tensors. Specifically, we establish blowup rates of both this singular potential and its gradient as $Q$ approaches its physical boundary.
All the proof is elementary.
\end{abstract}

\section{Introduction}

Liquid crystals are an intermediate state of matter between the
commonly observed solid and liquid that has no or partial
positional order but do exhibit an orientation order, and the simplest form of liquid crystals is called nematic type.
Broadly speaking, there are two types of models to describe nematic liquid crystals, namely the mean field model and the continuum model. In the former one,
the local alignment of liquid crystal molecules is described by a probability
distribution function on the unit sphere \cite{dG93,MS59,V94}.
Let $\n$ be a unit vector in $\RR^3$, representing the orientation of a single liquid crystal molecule, and $\rho(x; \n)$ be the density
distribution function of the orientation of all molecules at a point $x\in\Omega\subset\RR^3$.  The de Gennes $Q$-tensor, defined as the deviation of
the second moment of $\rho$ from its isotropic value, reads
\begin{equation}\label{def-physical-Q-tensor}
Q=\int_{\sS}\Big[\rho(\n)\otimes\rho(\n)-\frac{1}{3}\mathbb{I}_3\Big]\,\ud{\n}.
\end{equation}
Note that de Gennes $Q$-tensor vanishes in the isotropic phase, and hence it serves as
an order parameter. Meanwhile, it follows immediately from \eqref{def-physical-Q-tensor} that any de Gennes $Q$-tensor is symmetric, traceless, and all its eigenvalues
satisfy the constraint $-1/3\leq\lambda_i(Q)\leq 2/3$, $1\leq i\leq 3$.

In the continuum model, instead, a phenomenological Landau-de Gennes theory is proposed \cite{B12,dG93,MN14} such that
the alignment of liquid crystal molecules is described by the macroscopic $Q$-tensor order parameter, which is a symmetric, traceless $3\times 3$ matrix
without any eigenvalue constraint. In contrast with the de Gennes $Q$-tensor in the mean field model, this microscopic order parameter in the Landau-de Gennes
theory is at times referred to as the mathematical $Q$-tensor. In this framework the free energy functional is derived  as a nonlinear
integral functional of the $Q$-tensor and its spatial derivatives \cite{B12,M10}:
\begin{equation}
  \E[Q]=\int_{\Omega}\F(Q(x))\,dx,
\end{equation}
where $Q$ is the basic element in the so called $Q$-tensor space
\cite{B12}
\begin{equation}\label{Q-tensor-space}
\Q\defeq\Big\{M\in\mathbb{R}^{3\times 3}\;\Big|\; \tr(M)=0, \,M^{T}=M \Big\}.
\end{equation}
The free energy density functional $\F$ is composed of the elastic
part $\F_{\textit{el}}$ that depends on the gradient of $Q$, as well as the
bulk part $\F_{\textit{bulk}}$ that depends on $Q$ only. The bulk part $\F_{\textit{bulk}}$ is typically a
truncated expansion in the scalar invariants of the tensor $Q$ \cite{MZ10,PZ11,PZ12}
\begin{equation}\label{bulk energy}
\F_{\textit{bulk}}=\frac{a}{2}\tr(Q^2)-\frac{b}{3}\tr(Q^3)+\frac{c}{4}\tr^2(Q^2),
\end{equation}
where $a, b, c$ are assumed to be material-dependent coefficients. While
the simplest form of the elastic part $\F_{\textit{el}}$ that is invariant under rigid
rotations and material symmetry is \cite{B12, BM10, LMT87}
\begin{align}\label{elastic energy}
\F_{\textit{el}}=L_1|\nabla{Q}|^2+L_2\partial_jQ_{ik}\partial_kQ_{ij}+L_3\partial_jQ_{ij}\partial_kQ_{ik}
+L_4Q_{lk}\partial_kQ_{ij}\partial_lQ_{ij}.
\end{align}
Here, $\partial_kQ_{ij}$ stands for the $k$-th spatial derivative of the $ij$-th component of $Q$, $L_1,\cdots L_4$ are material dependent constants, and Einstein summation convention over repeated indices is used. It is noted that the retention of the $L_4$ cubic term is that it allows complete reduction to the classical Oseen-Frank energy
of liquid crystals with four elastic terms \cite{IXZ14}. On the other hand, however, this cubic term makes the free energy $\E[Q]$ unbounded from below \cite{BM10}.

To overcome this issue, a singular bulk potential $\psi_B$ is introduced in \cite{BM10} to replace the regular potential $\F_{\textit{bulk}}$. Specifically, the Ball-Majumdar
potential $f$ is defined by
\begin{equation}\label{def-singluar-potential}
f(Q)\defeq\begin{cases}
&\displaystyle\inf_{\rho\in\A_Q}\int_{\sS}\rho(\n)\ln\rho(\n)\,\ud\n, \quad -\dfrac13<\lambda_i(Q)<\frac23,\; 1\leq i\leq 3 \\
&+\infty, \qquad\qquad\qquad\qquad\qquad\qquad\qquad\mbox{otherwise},
\end{cases}
\end{equation}
where the admissible set $\A_{Q}$ is
\begin{equation}\label{admissible-set}
  \A_{Q}=\left\{\rho\in\mathcal{P}(\sS)\Big| \,\rho(\n)=\rho(-\n),\,\int_{\sS}\Big[\rho(\n)\otimes\rho(\n)-\frac{1}{3}\mathbb{I}_3\Big]\,\ud{\n}=Q\right\}.
\end{equation}
In other words, we minimize the Boltzmann entropy over all probability distributions $\rho$ with given normalized second moment $Q$.
Correspondingly,
\begin{equation}
\psi_B(Q)=f(Q)-\alpha|Q|^2
\end{equation}
is used to replace the commonly employed bulk potential $\F_{\textit{el}}$. Note that the last polynomial term is added to ensure the existence of local energy minimizers,
where $\alpha>0$ is a constant. As a consequence, $\psi_B$ imposes a natural enforcement of a physical constraint on
the eigenvalues of the mathematical $Q$-tensor. Further, the elastic energy part $\F_{\textit{el}}$ could be kept under control under mild assumptions on
$L_1, \cdots L_4$ \cite{DG98,IXZ14,KRSZ16}. Interested readers may also see \cite{GNS19} where a new Landau-de Gennes model with quartic elastic energy terms is proposed.

Analysis of this singular potential is undoubtedly not straightforward, and there has been some development in recent years. Concerning dynamic configurations, in a non-isothermal co-rotational Beris-Edwards system whose free energy consists of one elastic constant
term, namely $L_1$ term and this singular potential, the existence of global in time weak solutions is established in \cite{FRSZ14,FRSZ15}, and the convexity of $f$ is proved in \cite{FRSZ15}.
The existence, regularity and strict physicality of global weak solutions of the corresponding isothermal co-rotational Beris-Edwards system in a $2D$ torus is investigated in \cite{W12}, while global existence and partial regularity of a suitable weak solution to this system in $3D$ is established in \cite{DHW19}. The eigenvalue preservation of the co-rotational Beris-Edwards system with the regular bulk potential is studied in \cite{WXZ17} by virtue of $f$.
%\textcolor[rgb]{0.00,0.07,1.00}{Concerning an $L^2$ gradient flow generated by the free energy with $L_1, L_2, L_3$ terms and this singular potential, the global in time existence and uniqueness of solutions are established in \cite{LLX20}, wherein the partial regularity result
%of the global solution is also given}.
On the other hand, in static configurations, the H\"{o}lder regularity of global energy minimizer in $2D$ is established in \cite{BP16}, while partial regularity results for the global energy minimizer are given in \cite{EKT16}, and further improved in \cite{EKT16} under various assumptions of the blowup rates of $f$ and its gradient as $Q$ approaches its physical boundary. However, such assumptions are yet to be verified.

In static settings, the absolute minimizer of the free energy $\E$ satisfies the Euler-Lagrange equation
\begin{align}\label{elliptic-PDE}
&2L_1\Delta{Q}_{ij}+(L_2+L_3)\Big(\partial_k\partial_jQ_{ik}+\partial_k\partial_iQ_{jk}-\frac23\partial_k\partial_lQ_{kl}\delta_{ij}\Big)+2L_4\partial_k(Q_{lk}\partial_lQ_{ij})\non\\
&\qquad-L_4\partial_iQ_{kl}\partial_jQ_{kl}+\frac{L_4|\nabla{Q}|^2}{3}\delta_{ij}-\frac{\partial f}{\partial Q_{ij}}+\frac13\tr\Big(\frac{\partial f}{\partial Q}\Big)\delta_{ij}
+2\alpha{Q}_{ij}=0, \qquad 1\leq i, j\leq 3.
\end{align}
While in dynamic settings, a solution to an $L^2$ gradient flow generated by $\E$ satisfies
\begin{align}\label{parabolic-PDE}
\partial_tQ_{ij}&=2L_1\Delta{Q}_{ij}+(L_2+L_3)\Big(\partial_k\partial_jQ_{ik}+\partial_k\partial_iQ_{jk}-\frac23\partial_k\partial_lQ_{kl}\delta_{ij}\Big)+2L_4\partial_k(Q_{lk}\partial_lQ_{ij})\non\\
&\quad-L_4\partial_iQ_{kl}\partial_jQ_{kl}+\frac{L_4|\nabla{Q}|^2}{3}\delta_{ij}-\frac{\partial f}{\partial Q_{ij}}+\frac13\tr\Big(\frac{\partial f}{\partial Q}\Big)\delta_{ij}
+2\alpha{Q}_{ij}, \qquad 1\leq i, j\leq 3.
\end{align}
If $Q$ stays away from its physical boundary, then both $f$ and $\partial f/\partial{Q}$ are bounded functions. As a consequence,
under mild smallness assumption of $L_4$ both the elliptic problem \eqref{elliptic-PDE} and the parabolic problem \eqref{parabolic-PDE}
admit unique smooth solutions by direct methods in classical PDE theory. As $Q$ approaches its physical boundary, both the elliptic and parabolic equations become tensor-valued
variational obstacle problems, while both $f$ and $\partial f/\partial{Q}$ tends to
infinity. Therefore, it is an indispensable step to achieve their blowup rates for the corresponding PDE analysis in both the elliptic and the parabolic problems,
which is a fundamental issue to be solved.

Motivated by all the existing work, especially the aforementioned studies in both static and dynamic configurations, as well as future consideration of numeric approximations (see Remark \ref{remark for numerics} for details), in this paper we revisit the Ball-Majumdar potential $f$, and aim to establish the blowup rates of $f(Q)$, as well as its gradient $\nabla{f}(Q)$ near the physical boundary of $Q$.
In view of \eqref{def-singluar-potential}, here and after we always assume $Q$ is physical, in the sense that
\begin{equation}\label{physical}
-\frac13<\lambda_i(Q)<\frac23, \quad 1\leq i\leq 3.
\end{equation}

First, we provide a result regarding the blowup rate of $f(Q)$ as $Q$ approaches its physical boundary.
\begin{theorem}\label{theorem-upper-bound}
For any physical $Q$-tensor, assume $\lambda_1(Q)\leq\lambda_2(Q)\leq\lambda_3(Q)$. Then the functional $f$ defined in \eqref{def-singluar-potential} is bounded above by
\begin{equation}\label{upper-bound-new}
f(Q)\leq -\ln{8\sqrt{3}}-\frac12\ln\Big(\lambda_1(Q)+\frac13\Big)-\frac12\ln\Big(\lambda_2(Q)+\frac13\Big).
\end{equation}
Furthermore, there exists a small computable constant $\delta_0>0$,
{\color{black} such that,}
 whenever $Q$ approaches its physical boundary in the sense
that $\lambda_2(Q)+1/3<\delta_0$, it holds
{\color{black}
	\begin{equation}\label{lower-bound-new}
C_1-\frac12\ln\Big(\lambda_1(Q)+\frac13\Big)-\frac12\ln\Big(\lambda_2(Q)+\frac13\Big)\leq f(Q),\qquad C_1:=\ln{16}-8\ln\pi-\frac{\pi^5}{16}.
\end{equation}
}
\end{theorem}
\begin{remark}
It is noted that the result in Theorem \ref{theorem-upper-bound} is consistent with
\begin{equation}\label{rate-Ball}
\frac12\ln\Big[\frac{1}{(2\pi)^3e\big(\lambda_1(Q)+\frac13\big)}\Big]\leq f(Q)\leq\ln\Big[\frac{1}{\lambda_1(Q)+\frac13}\Big]
\end{equation}
obtained by Ball and Majumdar which is described in \cite{B18} and will appear in \cite{BM20}.
\end{remark}

\begin{remark}
Note that the upper bound in Theorem \ref{theorem-upper-bound} applies to any physical $Q$-tensor, while the lower bound is valid
when $\lambda_2(Q)$ gets close to $-1/3$ (which automatically implies $\lambda_1(Q)$ gets close to $-1/3$).
Theorem \ref{theorem-upper-bound} indicates that
as $Q$ approaches its physical boundary in the uniaxial direction
$$
  Q=\left(
    \begin{array}{ccc}
      -\dfrac13+\eps & 0 & 0 \\
      0 & -\dfrac13+\eps & 0 \\
      0 & 0 & \dfrac23-2\eps \\
    \end{array}
  \right), \qquad \eps\ll 1,
$$
$f(Q)$ blows up in the order of $-\ln(\lambda_1(Q)+1/3)$. Alternatively, when $Q$ is ``near" the uniaxial direction, that is, if $\lambda_2(Q)$ is close (but not equal) to $\lambda_1(Q)$, then any blowup order of $-\alpha\ln(\lambda_1(Q)+1/3)$, $1/2<\alpha<1$ could be attained.
On the other hand, when $\lambda_2(Q)$ stays away from $-1/3$,
$f(Q)$ is of the order $-1/2\ln(\lambda_1(Q)+1/3)$ as $\lambda_1(Q)$ approaches $-1/3$.
\end{remark}

\begin{remark}\label{remark for numerics}
Theorem \ref{theorem-upper-bound} will be of significance for numerics as well, because it implies that the function
$$
  f(Q) +\frac12\ln\Big(\lambda_1(Q)+\frac13\Big)+\frac12\ln\Big(\lambda_2(Q)+\frac 13 \Big)
$$
is a well defined, bounded function in the domain of $\lambda_1, \lambda_2$. Hence, by interpolating this well defined function, we can obtain an accurate numerical approximation of $f(Q)$.
\end{remark}

Moreover, the next theorem gives a precise blowup rate of $\nabla{f}$ near the physical boundary of $Q$.

\begin{theorem}\label{theorem-blowup-rate}
For any physical $Q$-tensor, assume $\lambda_1(Q)\leq\lambda_2(Q)\leq\lambda_3(Q)$. Then there exists a small computable constant $\eps_0>0$, 
{\color{black}such that,}
whenever
$\lambda_1(Q)+1/3<\eps_0$, the gradient of the functional $f$ defined in \eqref{def-singluar-potential} satisfies
\begin{equation}\label{blowup-rate}
\frac{C_1}{\lambda_1(Q)+\frac13}\leq \big|\nabla_{\Q}f(Q)\big|\leq\frac{C_2}{\lambda_1(Q)+\frac13},
\end{equation}
with the constants $C_1$ and $C_2$ given by
\begin{align}\label{constant-C1}
C_1=\frac{\sqrt{3}}{9\sqrt{2\pi}e}\cdot\inf_{\xi\geq 0}\frac{e^{-\xi}\IO(\xi)}{e^{\frac{-\xi}{2}}\IO(\frac{\xi}{2})}>0,\qquad
C_2=\sqrt{6\pi}e\cdot\sup_{\xi\geq 0}\frac{\exp\big(-\frac{\xi}{4}\big)\IO\big(\frac{\xi}{4}\big)}{\exp\big(-\frac{\xi}{2}\big)\IO\big(\frac{\xi}{2}\big)}.
\end{align}
Here
$$
\nabla_{\Q}f=\frac{\partial f}{\partial Q}-\frac13\tr\big(\frac{\partial f}{\partial Q}\big)\mathbb{I}_3,
$$
and $\IO(\cdot)$ is the zeroth order modified Bessel function of first kind.
\end{theorem}
\begin{remark}
We want to point out that \eqref{blowup-rate} is consistent with
\begin{equation}\label{blowup-rate-Ball}
  \frac12\sqrt{\frac32}\ln\Big[\frac{2}{\pi{e}\big(\lambda_1(Q)+\frac13\big)}\Big]\leq \big|\nabla_{\Q}f(Q)\big|\leq
  \frac{1}{\lambda_1(Q)+\frac13}\ln\Big[\frac{1}{2\pi^3{e}\big(\lambda_1(Q)+\frac13\big)}\Big],
\end{equation}
that is obtained by Ball and Majumdar in \cite{B18} and will also appear in \cite{BM20}.
\end{remark}
This paper is organized as follows. In Section $2$, we present a proof of Theorem \ref{theorem-upper-bound}.
Then in Section $3$, we give a proof of Theorem \ref{theorem-blowup-rate}.

\section{Blowup rate of $f$}

Note that \eqref{physical} is equivalent to $Q\in \mathcal{D}(f)$, namely the effective domain of $f$ where f assumes finite
values. As proved in \cite{FRSZ14}, $f$ is smooth for $Q\in \mathcal{D}(f)$.
Since $f$ is rotation invariant \cite{B12}, here and after, we always assume
that any considered physical $Q$-tensor
is diagonal:
\begin{equation}\label{Q-diagonal}
Q=\left(
    \begin{array}{ccc}
      \lambda_1 & 0 & 0 \\
      0 & \lambda_2 & 0 \\
      0 & 0 & \lambda_3 \\
    \end{array}
  \right), \qquad -\frac13<\lambda_1\leq\lambda_2\leq\lambda_3<\frac23,\;\lambda_1+\lambda_2+\lambda_3=0.
\end{equation}
Note that as $Q$ approaches its physical boundary, we have $\lambda_1\rightarrow -1/3$.

Correspondingly the optimal density function $\rho_Q\in\A_{Q}$ that satisfies $f(Q)=\int_{\sS}\rho_Q\ln\rho_Q\,\ud{S}$ is given by \cite{B12, BM10}
\begin{equation}\label{Boltzmann-distribution}
\rho_Q(x,y,z)=\dfrac{\exp(\mu_1x^2+\mu_2y^2+\mu_3z^2)}{Z(\mu_1,\mu_2,\mu_3)}, \quad(x,y,z)\in\sS,\quad\mu_1+\mu_2+\mu_3=0.
\end{equation}
Here in \eqref{Boltzmann-distribution}, $Z(\mu_1,\mu_2,\mu_3)$ is given by
\begin{equation}
Z(\mu_1,\mu_2,\mu_3)=\int_{\sS}\exp(\mu_1x^2+\mu_2y^2+\mu_3z^2)\,\ud{S},
\end{equation}
which satisfies
\begin{equation}\label{second-moment-original}
\frac{1}{Z}\frac{\partial Z}{\partial\mu_i}=\lambda_i+\frac13, \qquad 1\leq i\leq 3.
\end{equation}

To begin with, we have
\begin{lemma}\label{lemma-orders-mu}
For any physical $Q$-tensor \eqref{Q-diagonal}, its optimal probability density $\rho_Q$ defined in \eqref{Boltzmann-distribution}
satisfies
$$
  \mu_1\leq\mu_2\leq\mu_3.
$$
And $\mu_i=\mu_j$ provided $\lambda_i=\lambda_j$ for $1\leq i\neq j\leq 3$.
\end{lemma}
\begin{proof}
By symmetry, it suffices to prove that $\mu_i$ is strictly increasing in $\lambda_i$, i.e. $\mu_1<\mu_2$ whenever $\lambda_1<\lambda_2$.

\smallskip

Consider eigenvalues $\lambda_1<\lambda_2$.
Then it holds
\begin{equation}\label{inequality-second-moment}
\int_{\sS}x^2\rho_Q\,\ud{V}=\lambda_1+\frac13<\lambda_2+\frac13=\int_{\sS}y^2\rho_Q\,\ud{S}
\end{equation}
From \eqref{Boltzmann-distribution} we get
\begin{align*}
\rho_Q&=\dfrac{\exp\big\{\mu_1x^2+\mu_2y^2-(\mu_1+\mu_2)(1-x^2-y^2)\big\}}{\int_{\sS}\exp(\mu_1x^2+\mu_2y^2-(\mu_1+\mu_2)(1-x^2-y^2))\,\ud{S}}
=m^\ast\exp\big\{(2\mu_1+\mu_2)x^2+(\mu_1+2\mu_2)y^2\big\},
\end{align*}
where
\begin{equation}
m^\ast=\dfrac{1}{\int_{\sS}\exp\big\{(2\mu_1+\mu_2)x^2+(\mu_1+2\mu_2)y^2\big\}\,\ud{S}}
\end{equation}

Assume, by contradiction, that $\mu_1\geq\mu_2$. Using spherical coordinates
\begin{equation*}
\begin{cases}
x=\sin\theta\cos\phi\\
y=\sin\theta\sin\phi\\
z=\cos\theta
\end{cases}
\qquad 0\leq\phi< 2\pi,\; 0\leq\theta\leq\pi,
\end{equation*}
we get from \eqref{inequality-second-moment} that
\begin{align*}
\lambda_1&-\lambda_2 \\
&=\int_{\sS}x^2\rho_Q\,\ud{S}-\int_{\sS}y^2\rho_Q\,\ud{S}\\
&=8m^\ast\int_0^{\frac{\pi}{2}}\int_0^{\frac{\pi}{2}}(\cos^2\phi-\sin^2\phi)\exp\big\{(\mu_1-\mu_2)\sin^2\theta\cos^2\phi\big\}\,\ud\phi
\exp\big\{(\mu_1+2\mu_2)\sin^2\theta\big\}\sin^3\theta\,\ud\theta\\
&=8m^\ast\bigg[\int_0^{\frac{\pi}{2}}\int_0^{\frac{\pi}{4}}(\cos^2\phi-\sin^2\phi)\exp\big\{(\mu_1-\mu_2)\sin^2\theta\cos^2\phi\big\}\,\ud\phi
\exp\big\{(\mu_1+2\mu_2)\sin^2\theta\big\}\sin^3\theta\,\ud\theta\\
&\qquad+\int_0^{\frac{\pi}{2}}\underbrace{\int_{\frac{\pi}{4}}^{\frac{\pi}{2}}(\cos^2\phi-\sin^2\phi)\exp\big\{(\mu_1-\mu_2)\sin^2\theta\cos^2\phi\big\}\,\ud\phi}_{\psi=\frac{\pi}{2}-\phi}
\exp\big\{(\mu_1+2\mu_2)\sin^2\theta\big\}\sin^3\theta\,\ud\theta\bigg]\\
&=8m^\ast\int_0^{\frac{\pi}{2}}\int_0^{\frac{\pi}{4}}(\cos^2\phi-\sin^2\phi)\underbrace{\big(\exp\big\{(\mu_1-\mu_2)\sin^2\theta\cos^2\phi\big\}
-\exp\big\{(\mu_1-\mu_2)\sin^2\theta\sin^2\phi\big\}\big)}_{\geq 0}\,\ud\phi\\
&\qquad\exp\big\{(\mu_1+2\mu_2)\sin^2\theta\big\}\sin^3\theta\,\ud\theta\\
&\geq 0
\end{align*}
due to the assumption that $\mu_1\geq\mu_2$, which contradicts the fact that
$\lambda_1<\lambda_2$.
\end{proof}

Next we can see that the index $\mu_1$ in \eqref{Boltzmann-distribution} satisfies
\begin{lemma}\label{lemma-large-nu}
As $\lambda_1\rightarrow -\dfrac13$, $\mu_1\rightarrow -\infty$.
\end{lemma}
\begin{proof}
First, observe that
\begin{align*}
\frac{\partial\ln(\lambda_1+\frac13)}{\partial\mu_1}
&=\frac{\partial}{\partial\mu_1}\left[\ln\int_{\sS}x^2\exp(\mu_1x^2+\mu_2y^2+\mu_3z^2)\,\ud{S}-\ln{Z}(\mu_1, \mu_2, \mu_3)\right]\\
&=\frac{\int_{\sS}x^4\exp(\mu_1x^2+\mu_2y^2+\mu_3z^2)\,\ud{S}}{\int_{\sS}x^2\exp(\mu_1x^2+\mu_2y^2+\mu_3z^2)\,\ud{S}}-\Big(\lambda_1+\frac13\Big)\\
&=\Big(\lambda_1+\frac13\Big)^{-1}\bigg[\int_{\sS}x^4\rho_Q\,\ud{S} \underbrace{ \int_{\sS}\rho_Q\,\ud{S}}_{=1}
-\Big(\int_{\sS}x^2\rho_Q\,\ud{S}\Big)^2\bigg]>0
\end{align*}
due to Schwarz's inequality,
and the fact that $\rho_Q$ is not a perfect alignment of molecules
as $Q$ approaches the physical boundary.
Hence as $\lambda_1\searrow -1/3$, $\mu_1$ is
strictly decreasing. It remains to prove $\mu_1$ is unbounded as $\lambda_1\rightarrow -1/3$. Suppose there exists a constant $M>0$, such that
$\mu_1\geq -M$ as $\lambda_1\rightarrow -1/3$, then by \eqref{Boltzmann-distribution} and Lemma \ref{lemma-orders-mu} we see that $-M\leq\mu_1\leq\mu_2\leq\mu_3\leq 2M$. As a consequence, together with the basic inequality
\begin{equation}\label{trigonometry-inequality}
\frac{2\theta}{\pi}\leq\sin\theta<\theta, \quad\forall\, 0<\theta\leq\frac{\pi}{2},
\end{equation}
 we obtain
\begin{align*}
\lambda_1+\frac13&=\int_{\sS}x^2\rho_Q\,\ud{S}
=\frac{\int_{\sS}x^2\exp\big\{(2\mu_1+\mu_2)x^2+(\mu_1+2\mu_2)y^2\big\}\,\ud{S}}{\int_{\sS}\exp\big\{\underbrace{(2\mu_1+\mu_2)}_{\leq 0}x^2+\underbrace{(\mu_1+2\mu_2)}_{\leq 0}y^2\big\}\,\ud{S}} \\
&\geq\frac{\int_{\sS}x^2\exp\big\{-3M(x^2+y^2)\big\}\,\ud{S}}{\int_{\sS}\,\ud{S}}
=\frac{8\int_0^{\frac{\pi}{2}}\cos^2\phi\,\ud\phi\int_0^{\frac{\pi}{2}}\exp\big\{-3M\sin^2\theta\big\}\sin^3\theta\,\ud\theta}{4\pi} \\
&\geq\frac{2\pi\frac{8}{\pi^3}\int_0^{\frac{\pi}{2}}\exp(-3M\theta^2)\theta^3\,\ud\theta}{4\pi}
\geq\frac{4}{\pi^3}\exp\Big(\frac{-3M\pi^2}{4}\Big)\int_0^{\frac{\pi}{2}}\theta^3\,\ud\theta
=\frac{\pi}{16}\exp\Big(\frac{-3M\pi^2}{4}\Big),
\end{align*}
which is a contradiction.
Therefore, such lower bound $-M$ cannot exist, and the proof is complete.
\end{proof}

\begin{remark}\label{remark-small-eps}
It follows from the proof of Lemma \ref{lemma-large-nu} that in order to ensure $\mu_1<-M$ for any $M>0$, it suffices to
assume
$$
 \lambda_1(Q)+\frac13<\frac{\pi}{16}\exp\Big(\frac{-3M\pi^2}{4}\Big).
$$
\end{remark}

Now we are ready to prove Theorem \ref{theorem-upper-bound}.

\subsection{Proof of upper bound of $f$}

\begin{proof}

To this end, we consider $Q$ of the form
\begin{equation}\label{Q-example-1}
Q=\left(
    \begin{array}{ccc}
      -\dfrac13+\dfrac{\eps^2}{3} & 0 & 0 \\
      0 & \lambda_2 & 0 \\
      0 & 0 & \lambda_3 \\
    \end{array}
  \right), \qquad -\dfrac13+\frac{\eps^2}{3}\leq\lambda_2\leq\lambda_3,\;0<\eps\leq 1.
\end{equation}
Using the coordinate system
$$
\begin{cases}
x=\cos\theta \\
y=\sin\theta\sin\phi\\
z=\sin\theta\cos\phi
\end{cases}
\qquad 0\leq\phi< 2\pi,\;0\leq\theta\leq\pi,
$$
we consider the domain
\begin{equation}
S^\ast\defeq\big\{(1,\phi,\theta)\in\mathbb{S}^2\big| \;\phi\in [0, b]\cup[\pi-b,\pi]\cup[\pi,\pi+b]\cup[2\pi-b,2\pi],\;\theta\in [\arccos\eps,\pi-\arccos\eps] \big\},
\end{equation}
where $0<b\leq\pi/2$ is to be determined.
%\footnote{Note that when $b=\frac{\pi}{2}$, $Q_\eps$ is used to approximate
%$Q=\left(
%\begin{array}{ccc}
%1/6 & 0 & 0 \\
%0 & 1/6 & 0 \\
%0 & 0 & -1/3 \\
%\end{array}
%\right)
%$}
Meanwhile, let
$$
  \rho_\eps=\frac{1}{8b\eps}\chi_{S^\ast}.
$$
Then it is easy to check
\begin{align*}
\int_{\sS}\rho_\eps\,\ud{S}&=\frac{4}{8b\eps}\int_0^b\,\ud\phi\int_{\arccos\eps}^{\pi-\arccos\eps}\sin\theta\,\ud\theta=1,
\end{align*}
and the second moments with respect to $\rho_\eps$ are given by
\begin{align*}
\int_{\sS}x^2\rho_\eps\,\ud{S}&=\frac{1}{b\eps}\int_0^b\,\ud\phi\int_{\arccos\eps}^{\frac{\pi}{2}}\cos^2\theta\sin\theta\,\ud\theta=\frac{\eps^2}{3},\\
\int_{\sS}y^2\rho_\eps\,\ud{S}&=\frac{1}{b\eps}\int_0^b\sin^2\phi\,\ud\phi\int_{\arccos\eps}^{\frac{\pi}{2}}\sin^3\theta\,\ud\theta
=\frac{1}{b\eps}\Big(\frac{b}{2}-\frac{\sin{2b}}{4}\Big)\Big(\eps-\frac{\eps^3}{3}\Big)=\Big(\frac{1}{2}-\frac{\sin{2b}}{4b}\Big)\Big(1-\frac{\eps^2}{3}\Big),\\
\int_{\sS}z^2\rho_\eps\,\ud{S}&=\frac{1}{b\eps}\int_0^b\cos^2\phi\,\ud\phi\int_{\arccos\eps}^{\frac{\pi}{2}}\sin^3\theta\,\ud\theta
=\frac{1}{b\eps}\Big(\frac{b}{2}+\frac{\sin{2b}}{4}\Big)\Big(\eps-\frac{\eps^3}{3}\Big)=\Big(\frac{1}{2}+\frac{\sin{2b}}{4b}\Big)\Big(1-\frac{\eps^2}{3}\Big),\\
\int_{\sS}xy\rho_\eps\,\ud{S}&=\int_{\sS}yz\rho_\eps\,\ud{S}=\int_{\sS}zx\rho_\eps\,\ud{S}=0.
\end{align*}
Therefore, $\rho_\eps\in\mathcal{A}_{R}$ with
$$
  R=\left(
          \begin{array}{ccc}
            -\dfrac13+\dfrac{\eps^2}{3}  & 0 & 0 \\
            0 & \dfrac{1}{6}-\dfrac{\sin{2b}}{4b}-\big(\dfrac{1}{2}-\dfrac{\sin{2b}}{4b}\big)\dfrac{\eps^2}{3} & 0 \\
            0 & 0 & \dfrac{1}{6}+\dfrac{\sin{2b}}{4b}-\big(\dfrac{1}{2}+\dfrac{\sin{2b}}{4b}\big)\dfrac{\eps^2}{3} \\
          \end{array}
        \right)
$$
We need to find a suitable $0<b\leq \pi/2$, such that
$$
  \frac{1}{6}-\frac{\sin{2b}}{4b}-\Big(\frac{1}{2}-\frac{\sin{2b}}{4b}\Big)\frac{\eps^2}{3}=\lambda_2,
$$
which is equivalent to
\begin{equation}\label{b-equation}
\frac{\sin{2b}}{2b}=1-\frac{2\big(\lambda_2+\frac13\big)}{1-\frac{\eps^2}{3}}.
\end{equation}
%Note that the R.H.S. of \eqref{b-equation} is sufficiently close to $1$, while
Note that $\sin{(x)}/x$ is monotone decreasing in $(0, \pi]$, with
$$
  \lim_{x\rightarrow 0^+}\frac{\sin{x}}{x}=1,
$$
hence we know \eqref{b-equation} is solvable with $b\in (0, \pi/2]$. As a consequence, $R=Q$
and using mean value theorem we get
\begin{align*}
x-\frac{x^3}{6}<\sin{x}, \quad\forall\, 0<x\leq\pi; \qquad
\frac{1}{1-y}>1+y, \quad\forall\, 0<y<1.
\end{align*}
Hence after inserting $x=2b$, $y=\eps^2/3$ into \eqref{b-equation} we obtain
\begin{equation}
1-\frac{2b^2}{3}<\frac{\sin{2b}}{2b}<1-2\Big(1+\frac{\eps^2}{3}\Big)\Big(\lambda_2+\frac13\Big),
\end{equation}
which further implies
\begin{align}
b^2>(3+\eps^2)\Big(\lambda_2+\frac13\Big)>3\Big(\lambda_2+\frac13\Big).
\end{align}
In conclusion, we see $\rho_\eps\in\mathcal{A}_{Q}$ for $Q$ defined in \eqref{Q-example-1}, where $\lambda_{1}(Q)+1/3=\eps^2/3$, and
\begin{align}
f(Q)&\leq\int_{\sS}\rho_\eps\log\rho_\eps\,\ud{S}\leq\frac{8}{8b\eps}\int_0^b\,\ud\phi\int_{\arccos\eps}^{\frac{\pi}{2}}\ln\frac{1}{8b\eps}\sin\theta\,\ud\theta=\ln\frac{1}{8b\eps}
=-\ln{8}-\ln{b}-\ln{\eps} \nonumber\\
&\leq -\ln{8\sqrt3}-\frac12\ln\Big(\lambda_1(Q)+\frac13\Big)-\frac12\ln\Big(\lambda_2(Q)+\frac13\Big)\label{upper-bound-2}.
\end{align}

\end{proof}

\subsection{Proof of lower bound of $f$}
There is no doubt that the proof of lower bound \eqref{lower-bound-new} is far more difficult than that of \eqref{upper-bound-new}
{\color{black} because the latter requires just to construct one suitable
density $\rho\in \A_Q$, while the former requires us to treat all the
admissible densities in $\A_Q$.}
% because
%we can no longer utilize any specific probability density $\rho$ in the admissible set $\A_{Q}$. 
To accomplish the goal, more delicate analysis
is needed.

In this subsection we denote
\begin{equation}
  \eps=\lambda_1(Q)+\frac13, \qquad \delta=\lambda_2(Q)+\frac13.
\end{equation}
By \eqref{Boltzmann-distribution}-\eqref{second-moment-original}, we get
\begin{align}
\eps=\frac{\int_{\sS}x^2\exp(\mu_1x^2+\mu_2y^2+\mu_3z^2)\,\ud{S}}{\int_{\sS}\exp(\mu_1x^2+\mu_2y^2+\mu_3z^2)\,\ud{S}}
=\frac{\int_{\sS}x^2\exp(-\nu_1x^2-\nu_2y^2)\,\ud{S}}{\int_{\sS}\exp(-\nu_1x^2-\nu_2y^2)\,\ud{S}},\label{second-moment-new-1} \\
\delta=\frac{\int_{\sS}y^2\exp(\mu_1x^2+\mu_2y^2+\mu_3z^2)\,\ud{S}}{\int_{\sS}\exp(\mu_1x^2+\mu_2y^2+\mu_3z^2)\,\ud{S}}
=\frac{\int_{\sS}y^2\exp(-\nu_1x^2-\nu_2y^2)\,\ud{S}}{\int_{\sS}\exp(-\nu_1x^2-\nu_2y^2)\,\ud{S}},\label{second-moment-new-2}
\end{align}
where
\begin{equation}\label{def-nu-1}
\nu_1=-(2\mu_1+\mu_2),\quad
\nu_2=-(\mu_1+2\mu_2).
\end{equation}
By \eqref{Boltzmann-distribution} and Lemma \ref{lemma-orders-mu}, we see
\begin{equation}\label{positivity-nu-2}
\nu_1\geq\nu_2=\mu_3-\mu_2\geq 0.
\end{equation}
Besides, it follows from \eqref{Boltzmann-distribution}, Lemma \ref{lemma-orders-mu} and Lemma \ref{lemma-large-nu} that
\begin{equation}\label{large-nu-1}
\nu_1=-\mu_1+\mu_3\geq-\mu_1 \rightarrow +\infty, \quad\mbox{as }\;\eps\rightarrow 0.
\end{equation}
Actually, we can establish a stronger result in the following sense

\begin{lemma}\label{lemma-large-nu-2}
There exists a small computable constants $\delta_0>0$ such that
\begin{equation}
\delta\geq \frac{(2e-5)}{24e}e^{-\nu_2}, \qquad\forall\,\delta<\delta_0.
\end{equation}
As a consequence,
\begin{equation}
\nu_1\geq\nu_2\geq -\ln\delta+\ln\Big[\frac{(2e-5)}{24e}\Big], \qquad\forall\,\delta<\delta_0.
\end{equation}
\end{lemma}
\begin{proof}
Using \eqref{second-moment-new-2} and \eqref{positivity-nu-2} we have
\begin{equation}\label{delta-estimate}
\delta\geq\frac{\int_{\sS}y^2\exp(-\nu_1x^2-\nu_2)\,\ud{S}}{\int_{\sS}\exp(-\nu_1x^2)\,\ud{S}}
=e^{-\nu_2}\frac{\int_{\sS}y^2\exp(-\nu_1x^2)\,\ud{S}}{\int_{\sS}\exp(-\nu_1x^2)\,\ud{S}}.
\end{equation}
We proceed to estimate the numerator and denominator of R.H.S. in \eqref{delta-estimate}, respectively.  By Remark \ref{remark-small-eps},
$$
  \nu_1\geq-\mu_1>\frac{4}{\pi^2}, \qquad\mbox{provided }\;\delta<\frac{\pi}{16e^3}.
$$
Together with \eqref{trigonometry-inequality} and spherical coordinates
\begin{equation*}
\begin{cases}
x=\sin\theta\cos\phi, \\
y=\sin\theta\sin\phi, \\
z=\cos\theta,
\end{cases} \qquad 0\leq\theta\leq\pi,\;0\leq\phi<2\pi,
\end{equation*}
we have
\begin{align*}
\int_{\sS}\exp(-\nu_1x^2)\,\ud{S}&=2\int_0^{2\pi}\int_0^{\frac{\pi}{2}}\sin\theta e^{-\nu_1\sin^2\theta\cos^2\phi}\,\ud\theta\ud\phi
\leq 2\int_0^{2\pi}\int_0^{\frac{\pi}{2}}\theta e^{-\frac{4}{\pi^2}\nu_1\theta^2\cos^2\phi}\,\ud\theta\ud\phi \\
&=\frac{\pi^2}{4}\int_0^{2\pi}\frac{1-e^{-\nu_1\cos^2\phi}}{\nu_1\cos^2\phi}\,\ud\phi
=\pi^2\int_0^{\frac{\pi}{2}}\frac{1-e^{-\nu_1\sin^2\phi}}{\nu_1\sin^2\phi}\,\ud\phi \\
&\leq 4\underbrace{\int_0^{\frac{1}{\sqrt\nu_1}}\frac{1-e^{-\nu_1\phi^2}}{\nu_1\phi^2}\,\ud\phi}_{:=J_1}
+4\underbrace{\int_{\frac{1}{\sqrt\nu_1}}^{\frac{\pi}{2}}\frac{1}{\nu_1\phi^2}\,d\phi}_{:=J_2}.
\end{align*}
Note that
\begin{align*}
\sup_{\phi\in (0, 1/\sqrt\nu_1)}\frac{1-e^{-\nu_1\phi^2}}{\nu_1\phi^2}=1 \; \Longrightarrow\; J_1\leq \frac{1}{\sqrt\nu_1},
\quad\mbox{and } \quad J_2=\frac{\sqrt\nu_1-\frac{2}{\pi}}{\nu_1}\leq \frac{1}{\sqrt\nu_1}.
\end{align*}
Hence
\begin{equation}
\int_{\sS}\exp(-\nu_1x^2)\,\ud{S}\leq\frac{8}{\sqrt\nu_1}.
\end{equation}
Meanwhile, using \eqref{trigonometry-inequality} and integration by parts we obtain
\begin{align}\label{second-moment-y}
\int_{\sS}y^2\exp(-\nu_1x^2)\,\ud{S}&=2\int_0^{2\pi}\sin^2\phi\int_0^{\frac{\pi}{2}}\sin^3\theta e^{-\nu_1\sin^2\theta\cos^2\phi}\,\ud\theta\ud\phi\non\\
&\geq 2\Big(\frac{2}{\pi}\Big)^3\int_0^{2\pi}\sin^2\phi\int_0^{\frac{\pi}{2}}\theta^3e^{-\nu_1\theta^2\cos^2\phi}\,\ud\theta\ud\phi\non\\
&=\Big(\frac{4}{\pi}\Big)^3\int_0^{2\pi}\frac{\sin^2\phi}{2\nu_1^2\cos^4\phi}\Big[1-\Big(\frac{\pi^2}{4}\nu_1\cos^2\phi+1\Big)\exp\Big(-\frac{\pi^2}{4}\nu_1\cos^2\phi\Big)\Big]\,\ud\phi\non\\
&=\Big(\frac{4}{\pi}\Big)^3\int_0^{2\pi}\frac{\cos^2\phi}{2\nu_1^2\sin^4\phi}\Big[1-\Big(\frac{\pi^2}{4}\nu_1\sin^2\phi+1\Big)\exp\Big(-\frac{\pi^2}{4}\nu_1\sin^2\phi\Big)\Big]\,\ud\phi.
\end{align}
Note that
$$
\frac{\mathrm{d}}{\mathrm{d}z}\Big[\Big(\frac{\pi^2}{4}\nu_1z+1\Big)e^{-\frac{\pi^2}{4}\nu_1z}\Big]=-\frac{\pi^2}{4}\nu_1ze^{-\frac{\pi^2}{4}\nu_1z}
\leq 0, \qquad\forall\, z\geq 0.
$$
Hence
$$
  1-\Big(\frac{\pi^2}{4}\nu_1\sin^2\phi+1\Big)\exp\Big(-\frac{\pi^2}{4}\nu_1\sin^2\phi\Big)\geq 0, \qquad\forall\,\phi\in [0, 2\pi),
$$
and \eqref{second-moment-y} continues as
\begin{align}
\int_{\sS}y^2\exp(-\nu_1x^2)\,\ud{S}&\geq\Big(\frac{4}{\pi}\Big)^3
\int_{\frac{2}{\pi\sqrt\nu_1}}^{\frac{\pi}{4}}\frac{\cos^2\phi}{2\nu_1^2\sin^4\phi}\Big[1-\Big(\frac{\pi^2}{4}\nu_1\sin^2\phi+1\Big)\exp\Big(-\frac{\pi^2}{4}\nu_1\sin^2\phi\Big)\Big]\,\ud\phi\non\\
&\geq\Big(\frac{4}{\pi}\Big)^3
\int_{\frac{2}{\pi\sqrt\nu_1}}^{\frac{\pi}{4}}\frac{1}{4\nu_1^2\sin^4\phi}\Big[1-\Big(\frac{\pi^2}{4}\nu_1\sin^2\phi+1\Big)\exp\Big(-\frac{\pi^2}{4}\nu_1\sin^2\phi\Big)\Big]\,\ud\phi\non\\
&\geq\Big(\frac{4}{\pi}\Big)^3\int_{\frac{2}{\pi\sqrt\nu_1}}^{\frac{\pi}{4}}\frac{1}{4\nu_1^2\sin^4\phi}\Big[1-\Big(\frac{\pi^2}{4}\nu_1\sin^2\frac{2}{\pi\sqrt\nu_1}+1\Big)
\exp\Big(-\frac{\pi^2}{4}\nu_1\sin^2\frac{2}{\pi\sqrt\nu_1}\Big)\Big]\,\ud\phi\non\\
&\geq\Big(\frac{4}{\pi}\Big)^3\int_{\frac{2}{\pi\sqrt\nu_1}}^{\frac{\pi}{4}}\frac{1}{4\nu_1^2\sin^4\phi}\Big[1-2\exp\Big(-\frac{\pi^2}{4}\nu_1\sin^2\frac{2}{\pi\sqrt\nu_1}\Big)\Big]\,\ud\phi.
\end{align}
By \eqref{large-nu-1},
$$
  \lim_{\delta\rightarrow 0}\exp\Big(-\frac{\pi^2}{4}\nu_1\sin^2\frac{2}{\pi\sqrt\nu_1}\Big)=\lim_{\nu_1\rightarrow\infty}\exp\Big(-\frac{\pi^2}{4}\nu_1\sin^2\frac{2}{\pi\sqrt\nu_1}\Big)=\frac{1}{e}.
$$
Thus by Remark \ref{remark-small-eps} there exists a small computable constant $\delta_0>0$, such that
$$
  1-2\exp\Big(-\frac{\pi^2}{4}\nu_1\sin^2\frac{2}{\pi\sqrt\nu_1}\Big)\geq 1-\frac{5}{2e}, \qquad \forall\,\delta<\delta_0.
$$
This  further implies
\begin{align}
\int_{\sS}y^2\exp(-\nu_1x^2)\,\ud{S}&\geq\Big(1-\frac{5}{2e}\Big)\Big(\frac{4}{\pi}\Big)^3\int_{\frac{2}{\pi\sqrt\nu_1}}^{\frac{\pi}{4}}\frac{1}{4\nu_1^2\sin^4\phi}\,\ud\phi
\geq \Big(1-\frac{5}{2e}\Big)\Big(\frac{4}{\pi}\Big)^3\int_{\frac{2}{\pi\sqrt\nu_1}}^{\frac{\pi}{4}}\frac{1}{4\nu_1^2\phi^4}\,\ud\phi \non\\
&\geq \frac{2e-5}{3e\sqrt\nu_1}, \qquad \forall\,\delta<\delta_0.
\end{align}
To sum up, we conclude
$$
  \delta\geq e^{-\nu_2}\frac{\int_{\sS}y^2\exp(-\nu_1x^2)\,\ud{S}}{\int_{\sS}\exp(-\nu_1x^2)\,\ud{S}}
  \geq e^{-\nu_2}\frac{(2e-5)}{3e\sqrt\nu_1}\Big/\frac{8}{\sqrt\nu_1}=\frac{(2e-5)}{24e}e^{-\nu_2}, \qquad\forall\,\delta<\delta_0,
$$
completing the proof.
\end{proof}
\begin{remark}
Lemma \ref{lemma-large-nu-2} implies that as $\lambda_2(Q)\rightarrow -1/3$, that is $\delta\rightarrow 0^+$, we have both
$\nu_1, \nu_2\rightarrow +\infty$.
\end{remark}

Next, we denote
\begin{equation}
A(\phi)=\nu_1\cos^2\phi+\nu_2\sin^2\phi, \qquad 0\leq\phi<2\pi.
\end{equation}
We can establish the following estimates

\begin{lemma}\label{lemma-second-moment-estimate}
There exists a small computable constant $\delta_0>0$ such that
\begin{align}
\frac{1}{\pi}\int_0^{2\pi}\frac{1}{A(\phi)}\,\ud\phi&\leq\int_{\sS}\exp(-\nu_1x^2-\nu_2y^2)\,\ud{S}\leq\frac{\pi^2}{4}\int_0^{2\pi}\frac{1}{A(\phi)}\,\ud\phi,
\qquad\qquad\forall\,\delta<\delta_0,\label{lower-bound-denominator}   \\
\frac{4}{\pi^3}\int_0^{2\pi}\frac{\cos^2\phi}{A^2(\phi)}\,\ud\phi&\leq\int_{\sS}x^2\exp(-\nu_1x^2-\nu_2y^2)\,\ud{S}
\leq\frac{\pi^4}{16}\int_0^{2\pi}\frac{\cos^2\phi}{A^2(\phi)}\,\ud\phi, \qquad\forall\,\delta<\delta_0,\label{lower-bound-numerator-1}   \\
\frac{4}{\pi^3}\int_0^{2\pi}\frac{\sin^2\phi}{A^2(\phi)}\,\ud\phi&\leq\int_{\sS}y^2\exp(-\nu_1x^2-\nu_2y^2)\,\ud{S}
\leq\frac{\pi^4}{16}\int_0^{2\pi}\frac{\sin^2\phi}{A^2(\phi)}\,\ud\phi, \qquad\forall\,\delta<\delta_0. \label{lower-bound-numerator-2}
\end{align}
\end{lemma}
\begin{proof}
To begin with, by \eqref{positivity-nu-2} we have
$$
A(\phi)\geq\nu_2\cos^2\phi+\nu_2\sin^2\phi=\nu_2.
$$
This together with Lemma \ref{lemma-large-nu-2} implies that there exists a small computable constant $\delta_0>0$ such that
\begin{equation}\label{exponential-decay}
e^{\frac{-\pi^2A(t)}{4}}\leq e^{\frac{-\pi^2\nu_2}{4}}\leq\frac12-\frac{1}{e},\qquad\forall\,\delta<\delta_0.
\end{equation}
Using \eqref{trigonometry-inequality} and integration by parts, we obtain
\begin{align*}
\int_{\sS}\exp(-\nu_1x^2-\nu_2y^2)\,\ud{S}&=2\int_0^{2\pi}\int_0^{\frac{\pi}{2}}\sin\theta\exp\big[-A(\phi)\sin^2\theta\big]\,\ud\theta\ud\phi\\
&\leq 2\int_0^{2\pi}\int_0^{\frac{\pi}{2}}\theta\exp\big[-4A(\phi)\pi^{-2}\theta^2\big]\,\ud\theta\ud\phi\\
&=\int_0^{2\pi}\frac{\pi^2-\pi^2e^{-A(\phi)}}{4A(\phi)}\,\ud\phi\leq\frac{\pi^2}{4}\int_0^{2\pi}\frac{1}{A(\phi)}\,\ud\phi,
\end{align*}
and together with \eqref{exponential-decay} we have
\begin{align*}
\int_{\sS}\exp(-\nu_1x^2-\nu_2y^2)\,\ud{S}&=2\int_0^{2\pi}\int_0^{\frac{\pi}{2}}\sin\theta\exp\big[-A(\phi)\sin^2\theta\big]\,\ud\theta\ud\phi\\
&\geq \frac{4}{\pi}\int_0^{2\pi}\int_0^{\frac{\pi}{2}}\theta\exp\big[-A(\phi)\theta^2\big]\,\ud\theta\ud\phi\\
&=\frac{2}{\pi}\int_0^{2\pi}\frac{1}{A(\phi)}\big[1-e^{-A(\phi)\pi^2/4}\big]\,\ud\phi\\
&\geq \frac{1}{\pi}\int_0^{2\pi}\frac{1}{A(\phi)}\,\ud\phi,
\end{align*}
concluding the proof of \eqref{lower-bound-denominator}.

\smallskip

To proceed, using \eqref{trigonometry-inequality} and integration by parts again, we get
\begin{align*}
\int_{\sS}x^2\exp(-\nu_1x^2-\nu_2y^2)\,\ud{S}&=2\int_0^{2\pi}\int_0^{\frac{\pi}{2}}\cos^2\phi\sin^3\theta\exp\big[-A(\phi)\sin^2\theta\big]\,\ud\theta\ud\phi\\
&\leq 2\int_0^{2\pi}\int_0^{\frac{\pi}{2}}\cos^2\phi\;\theta^3\exp\big[-A(\phi)4\pi^{-2}\theta^2\big]\,\ud\theta\ud\phi\\
&=\frac{\pi^4}{16}\int_0^{2\pi}\frac{\cos^2\phi}{A^2(\phi)}\Big\{1-e^{-A(\phi)}\big[A(\phi)+1\big] \Big\}\,\ud\phi \\
&\leq\frac{\pi^4}{16}\int_0^{2\pi}\frac{\cos^2\phi}{A^2(\phi)}\,\ud\phi
\end{align*}
and similarly
\begin{align*}
&\int_{\sS}x^2\exp(-\nu_1x^2-\nu_2y^2)\,\ud{S}\\
&=2\int_0^{2\pi}\int_0^{\frac{\pi}{2}}\cos^2\phi\sin^3\theta\exp\big[-A(\phi)\sin^2\theta\big]\,\ud\theta\ud\phi\\
&\geq 2\Big(\frac{2}{\pi}\Big)^3\int_0^{2\pi}\int_0^{\frac{\pi}{2}}\cos^2\phi\;\theta^3\exp\big[-A(\phi)\theta^2\big]\,\ud\theta\ud\phi\\
&=\Big(\frac{2}{\pi}\Big)^3\int_0^{2\pi}\frac{\cos^2\phi}{A^2(\phi)}\Big\{1-\frac14e^{-\pi^2A(\phi)/4}\big[\pi^2A(\phi)+4\big] \Big\}\,\ud\phi\\
&=\frac{8}{\pi^3}\int_0^{2\pi}\frac{\cos^2\phi}{A^2(\phi)}\,\ud\phi
-\frac{8}{\pi^3}\underbrace{\int_0^{2\pi}\frac{\cos^2\phi}{A^2(\phi)}\frac{\pi^2A(\phi)}{4}\exp\Big[\frac{-\pi^2A(\phi)}{4}\Big]\,\ud\phi}_{I_1}
-\frac{8}{\pi^3}\underbrace{\int_0^{2\pi}\frac{\cos^2\phi}{A^2(\phi)}\exp\Big[\frac{-\pi^2A(\phi)}{4}\Big]\,\ud\phi}_{I_2} {\color{black}.}
\end{align*}
Hence to attain \eqref{lower-bound-numerator-1}, it remains to estimate $I_1$ and $I_2$. First, observe that $ze^{-z}\in [0, 1/e]$ for $z\geq 0$, hence
$$
  I_1\leq \frac{1}{e}\int_0^{2\pi}\frac{\cos^2\phi}{A^2(\phi)}\,\ud\phi.
$$
Besides, it follows from \eqref{exponential-decay} that
$$
  I_2\leq \Big(\frac12-\frac{1}{e}\Big)\int_0^{2\pi}\frac{\cos^2\phi}{A^2(\phi)}\,\ud\phi.
$$
To sum up, we conclude
$$
  \int_{\sS}x^2\exp(-\nu_1x^2-\nu_2y^2)\,\ud{S}\geq \frac{4}{\pi^3}\int_0^{2\pi}\frac{\cos^2\phi}{A^2(\phi)}\,\ud\phi,
$$
hence the proof of \eqref{lower-bound-numerator-1} is complete. The proof of \eqref{lower-bound-numerator-2} is completely analogous to that of
\eqref{lower-bound-numerator-1}.
\end{proof}

As a matter of fact, all the bounds in Lemma \ref{lemma-second-moment-estimate} can be achieved explicitly in terms of $\nu_1, \nu_2$.

\begin{lemma}\label{lemma-antiderivative}
For any $\nu_1, \nu_2>0$, the following identities are satisfied:
\begin{align}
\int_0^{2\pi}\frac{1}{A(\phi)}\,\ud\phi&=4\int_0^{\frac{\pi}{2}}\frac{1}{A(\phi)}\,\ud\phi=\frac{2\pi}{\sqrt{\nu_1\nu_2}}, \\
\int_0^{2\pi}\frac{\sin^2\phi}{A^2(\phi)}\,\ud\phi&=4\int_0^{\frac{\pi}{2}}\frac{\sin^2\phi}{A^2(\phi)}\,\ud\phi=\frac{\pi}{\nu_2\sqrt{\nu_1\nu_2}},\\
\int_0^{2\pi}\frac{\cos^2\phi}{A^2(\phi)}\,\ud\phi&=4\int_0^{\frac{\pi}{2}}\frac{\cos^2\phi}{A^2(\phi)}\,\ud\phi=\frac{\pi}{\nu_1\sqrt{\nu_1\nu_2}}.
\end{align}
\end{lemma}
\begin{proof}
The proof relies on direct derivation of these anti-derivatives. Note that
\begin{align*}
\int\frac{1}{A(\phi)}\,\ud\phi&=\int\frac{\sec^2\phi}{\nu_1+\nu_2\tan^2\phi}\,\ud\phi
=\frac{1}{\sqrt{\nu_1\nu_2}}\int\frac{1}{1+\big(\sqrt{\frac{\nu_2}{\nu_1}}\tan\phi\big)^2}\,\ud\Big(\sqrt{\frac{\nu_2}{\nu_1}}\tan\phi\Big)\\
&=\frac{1}{\sqrt{\nu_1\nu_2}}\arctan\Big(\sqrt{\frac{\nu_2}{\nu_1}}\tan\phi\Big)+C.
\end{align*}
Hence
$$
\int_0^{2\pi}\frac{1}{A(\phi)}\,\ud\phi=4\int_0^{\frac{\pi}{2}}\frac{1}{A(\phi)}\,\ud\phi
=\frac{4}{\sqrt{\nu_1\nu_2}}\arctan\Big(\sqrt{\frac{\nu_2}{\nu_1}}\tan\phi\Big)\Big|_{0}^{\frac{\pi}{2}}=\frac{2\pi}{\sqrt{\nu_1\nu_2}}.
$$
Next,
\begin{align*}
&\int\frac{\sin^2\phi}{A^2(\phi)}\,\ud\phi\\
&=\int\frac{\tan^2\phi\sec^2\phi}{(\nu_1+\nu_2\tan^2\phi)^2}\,\ud\phi
=\frac{1}{\nu_2}\int\frac{(\nu_1+\nu_2\tan^2\phi)\sec^2\phi}{(\nu_1+\nu_2\tan^2\phi)^2}\,\ud\phi-\frac{\nu_1}{\nu_2}\int\frac{\sec^2\phi}{(\nu_1+\nu_2\tan^2\phi)^2}\,\ud\phi\\
&=\frac{1}{\sqrt{\nu_1\nu_2}\nu_2}\int\frac{1}{1+\big(\sqrt{\frac{\nu_2}{\nu_1}}\tan\phi\big)^2}\,\ud\Big(\sqrt{\frac{\nu_2}{\nu_1}}\tan\phi\Big)
-\frac{1}{\sqrt{\nu_1\nu_2}\nu_2}\int\frac{1}{\Big[1+\big(\sqrt{\frac{\nu_2}{\nu_1}}\tan\phi\big)^2\Big]^2}\,\ud\Big(\sqrt{\frac{\nu_2}{\nu_1}}\tan\phi\Big)\\
&=\frac{1}{\sqrt{\nu_1\nu_2}\nu_2}\arctan\Big(\sqrt{\frac{\nu_2}{\nu_1}}\tan\phi\Big)
-\frac{1}{\sqrt{\nu_1\nu_2}\nu_2}\underbrace{\int\frac{1}{\Big[1+\big(\sqrt{\frac{\nu_2}{\nu_1}}\tan\phi\big)^2\Big]^2}\,\ud\Big(\sqrt{\frac{\nu_2}{\nu_1}}\tan\phi\Big)}_{I_1}.
\end{align*}
By setting $u=\sqrt{\frac{\nu_2}{\nu_1}}\tan\phi$ we obtain
\begin{align*}
I_1&=\int\frac{1}{(1+u^2)^2}\,\ud{u}\overset{(\theta=\tan u)}{=}\int\cos^2\theta\,\ud\theta=\frac{\sin(2\theta)}{4}+\frac{\theta}{2}+C=\frac12\frac{u}{1+u^2}+\frac{\arctan{u}}{2}+C.
\end{align*}
Thus
\begin{align*}
\int\frac{\sin^2\phi}{A^2(\phi)}\,\ud\phi=\frac{1}{2\sqrt{\nu_1\nu_2}\nu_2}\arctan\Big(\sqrt{\frac{\nu_2}{\nu_1}}\tan\phi\Big)
-\frac{1}{2\nu_2}\frac{\cos\phi\sin\phi}{\nu_1\cos^2\phi+\nu_2\sin^2\phi}+C,
\end{align*}
and henceforth
\begin{align*}
\int_0^{2\pi}\frac{\sin^2\phi}{A^2(\phi)}\,\ud\phi=4\int_0^{\frac{\pi}{2}}\frac{\sin^2\phi}{A^2(\phi)}\,\ud\phi
=\frac{2}{\sqrt{\nu_1\nu_2}\nu_2}\arctan\Big(\sqrt{\frac{\nu_2}{\nu_1}}\tan\phi\Big)\bigg|_0^{\frac{\pi}{2}}=\frac{\pi}{\sqrt{\nu_1\nu_2}\nu_2}.
\end{align*}
Similarly,
\begin{align*}
\int\frac{\cos^2\phi}{A^2(\phi)}\,\ud\phi&=\int\frac{\sec^2\phi}{(\nu_1+\nu_2\tan^2\phi)^2}\,\ud\phi
=\frac{1}{\nu_1\sqrt{\nu_1\nu_2}}\underbrace{\int\frac{1}{\Big[1+\big(\sqrt{\frac{\nu_2}{\nu_1}}\tan\phi\big)^2\Big]^2}\,\ud\Big(\sqrt{\frac{\nu_2}{\nu_1}}\tan\phi\Big)}_{I_1}\\
&=\frac{1}{2\sqrt{\nu_1\nu_2}\nu_1}\arctan\Big(\sqrt{\frac{\nu_2}{\nu_1}}\tan\phi\Big)+\frac{1}{2\nu_1}\frac{\cos\phi\sin\phi}{\nu_1\cos^2\phi+\nu_2\sin^2\phi}+C.
\end{align*}
Hence
\begin{align*}
\int_0^{2\pi}\frac{\cos^2\phi}{A^2(\phi)}\,\ud\phi=4\int_0^{\frac{\pi}{2}}\frac{\cos^2\phi}{A^2(\phi)}\,\ud\phi
=\frac{2}{\sqrt{\nu_1\nu_2}\nu_1}\arctan\Big(\sqrt{\frac{\nu_2}{\nu_1}}\tan\phi\Big)\bigg|_0^{\frac{\pi}{2}}=\frac{\pi}{\sqrt{\nu_1\nu_2}\nu_1}.
\end{align*}

\end{proof}

By virtue of Lemma \ref{lemma-large-nu-2}, Lemma \ref{lemma-second-moment-estimate}, and Lemma \ref{lemma-antiderivative}, we are ready to prove the lower bound \eqref{lower-bound-new}
in Theorem \ref{theorem-upper-bound}.

\begin{proof}[Proof of \eqref{lower-bound-new}]
Let $\delta_0>0$ be the minimum threshold from the previous lemmas.
By \eqref{second-moment-new-1}, \eqref{second-moment-new-2}, Lemma \ref{lemma-second-moment-estimate}, and Lemma \ref{lemma-antiderivative} we have
\begin{align}
\eps&=\frac{\int_{\sS}x^2\exp(-\nu_1x^2-\nu_2y^2)\,\ud{S}}{\int_{\sS}\exp(-\nu_1x^2-\nu_2y^2)\,\ud{S}}
\leq\dfrac{\frac{\pi^4}{16}\int_0^{2\pi}\frac{\cos^2\phi}{A^2(\phi)}\,\ud\phi}{\frac{1}{\pi}\int_0^{2\pi}\frac{1}{A(\phi)}\,\ud\phi}=\frac{\pi^5}{32}\frac{1}{\nu_1},
\label{key-bound-1}\\
\eps&=\frac{\int_{\sS}x^2\exp(-\nu_1x^2-\nu_2y^2)\,\ud{S}}{\int_{\sS}\exp(-\nu_1x^2-\nu_2y^2)\,\ud{S}}
\geq\dfrac{\frac{4}{\pi^3}\int_0^{2\pi}\frac{\cos^2\phi}{A^2(\phi)}\,\ud\phi}{\frac{\pi^2}{4}\int_0^{2\pi}\frac{1}{A(\phi)}\,\ud\phi}=\frac{8}{\pi^5}\frac{1}{\nu_1},
\label{key-bound-2}\\
\delta&=\frac{\int_{\sS}y^2\exp(-\nu_1x^2-\nu_2y^2)\,\ud{S}}{\int_{\sS}\exp(-\nu_1x^2-\nu_2y^2)\,\ud{S}}
\leq\dfrac{\frac{\pi^4}{16}\int_0^{2\pi}\frac{\sin^2\phi}{A^2(\phi)}\,\ud\phi}{\frac{1}{\pi}\int_0^{2\pi}\frac{1}{A(\phi)}\,\ud\phi}=\frac{\pi^5}{32}\frac{1}{\nu_2},
\label{key-bound-3}\\
\delta&=\frac{\int_{\sS}y^2\exp(-\nu_1x^2-\nu_2y^2)\,\ud{S}}{\int_{\sS}\exp(-\nu_1x^2-\nu_2y^2)\,\ud{S}}
\geq\dfrac{\frac{4}{\pi^3}\int_0^{2\pi}\frac{\sin^2\phi}{A^2(\phi)}\,\ud\phi}{\frac{\pi^2}{4}\int_0^{2\pi}\frac{1}{A(\phi)}\,\ud\phi}=\frac{8}{\pi^5}\frac{1}{\nu_2}.
\label{key-bound-4}
\end{align}
These estimates \eqref{key-bound-1}-\eqref{key-bound-4} together with \eqref{Boltzmann-distribution}, \eqref{def-nu-1}, Lemma \ref{lemma-second-moment-estimate}, and Lemma \ref{lemma-antiderivative} yield
\begin{align}\label{estimate-f}
f(Q)&=-\ln{Z}+\sum_{i=1}^3\mu_i\big(\lambda_i+\frac13\big)=-\ln{Z}+\eps\mu_1+\delta\mu_2+(1-\eps-\delta)(-\mu_1-\mu_2)\non\\
&=-\ln{Z}-\eps\nu_1-\delta\nu_2-(\mu_1+\mu_2) \non\\
&=-\ln\Big[\int_{\sS}\exp(\mu_1+\mu_2)\exp(\mu_1x^2+\mu_2y^2+\mu_3z^2)\,\ud{S}\Big]-\eps\nu_1-\delta\nu_2 \non\\
&=-\ln\int_{\sS}\exp(-\nu_1x^2-\nu_2y^2)\,\ud{S}-\eps\nu_1-\delta\nu_2=-\ln\Big[\frac{\pi^2}{4}\int_0^{2\pi}\frac{1}{A(\phi)}\,\ud\phi\Big]-\eps\nu_1-\delta\nu_2\non\\
&\geq\frac12\ln(\nu_1\nu_2)-\ln{2\pi}-\ln\frac{\pi^2}{4}-\frac{\pi^5}{16}\non\\
&\geq -\frac12\ln(\eps\delta)+\ln{16}-8\ln\pi-\frac{\pi^5}{16}, \qquad\qquad\forall\,\delta<\delta_0,
\end{align}
concluding the proof of \eqref{lower-bound-new}.
\end{proof}

\begin{remark}
Actually, \eqref{estimate-f} also provides a proof of the upper bound of the order
$$
  C-\frac12\ln\Big(\lambda_1(Q)+\frac13\Big)-\frac12\ln\Big(\lambda_2(Q)+\frac13\Big)
$$
for an explicit constant $C$. However, it requires $\lambda_2(Q)$ to be sufficiently close to $-1/3$, which is not necessary in the argument regarding the upper bound of $f$
provided in the previous subsection.
\end{remark}

\section{Blowup rate of $\nabla_{\Q}{f}$}

In this section we shall finish proof of Theorem \ref{theorem-blowup-rate}, which plays a crucial role
in the proofs regarding regularity results of the relevant solutions
in both static and dynamic configurations \cite{EKT16, GT19, LLX20}. The argument in the proof of \eqref{lower-bound-new} is no longer valid here
in that now we only assume $\lambda_1(Q)$ is sufficiently close to $-1/3$.

\smallskip

To this end, we shall compute all five components of $\nabla_{\Q}f(Q)$, and state that only the ``radial" and ``tangential" components of $\nabla_{\Q}f$ are nonzero. Then we will be focused on
the estimate of its ``radial" component only.

\begin{lemma}\label{lemma-rate-gradient}
For any physical $Q$-tensor of the form \eqref{Q-diagonal}, it holds
\begin{align}
\nabla_{rad}f(Q)&=\sqrt{\frac23}\frac{\mathrm{d}{f}(Q_\eps^\bot)}{\mathrm{d}\eps}\bigg|_{\eps=0}=-\sqrt{\frac32}\mu_1,\label{radial-component}\\
\nabla_{tan}f(Q)&=\sqrt{\frac12}\frac{\mathrm{d}{f}(Q_\eps^\|)}{\mathrm{d}\eps}\bigg|_{\eps=0}=\sqrt{\frac12}(\mu_1+2\mu_2).\label{tangential-component}
\end{align}
Here $\nabla_{rad}f(Q)$ (resp. $\nabla_{tan}f(Q)$) is defined in \eqref{radial-def} (resp. in \eqref{tangential-def}) below.
\end{lemma}
\begin{proof}
Recall from \cite[Lemma~C.1]{GT19} that for a given physical $Q$-tensor of the form \eqref{Q-diagonal},
its projection on the physical boundary is
$$
  Q^\bot=\left(\begin{array}{ccc}
           -\frac13 & 0 & 0 \\
           0 & \lambda_2+\frac{\lambda_1+\frac13}{2} & 0 \\
           0 & 0 & \lambda_3+\frac{\lambda_1+\frac13}{2}
         \end{array}\right),
$$
and their distance is
$$
  d(Q)\defeq |Q-Q^\bot|=\frac{\sqrt{6}}{2}\big(\lambda_1+\frac13\big).
$$
Let us introduce
\begin{align*}
Q_\eps^{(1)}\defeq\left(\begin{array}{ccc}
           \lambda_1-\eps & 0 & 0 \\
           0 & \lambda_2+\frac{\eps}{2} & 0 \\
           0 & 0 & \lambda_3+\frac{\eps}{2}
         \end{array}\right),\qquad
Q_\eps^{(2)}\defeq\left(\begin{array}{ccc}
           \lambda_1 & 0 & 0 \\
           0 & \lambda_2+\eps & 0 \\
           0 & 0 & \lambda_3-\eps
         \end{array}\right), \\
Q_\eps^{(3)}\defeq\left(\begin{array}{ccc}
           \lambda_1 & \eps & 0 \\
           \eps & \lambda_2 & 0 \\
           0 & 0 & \lambda_3
         \end{array}\right),\qquad
Q_\eps^{(4)}\defeq\left(\begin{array}{ccc}
           \lambda_1 & 0 & \eps \\
           0 & \lambda_2 & 0 \\
           \eps & 0 & \lambda_3
         \end{array}\right),\qquad
Q_\eps^{(5)}\defeq\left(\begin{array}{ccc}
           \lambda_1 & 0 & 0 \\
           0 & \lambda_2 & \eps \\
           0 & \eps & \lambda_3
         \end{array}\right).
\end{align*}
And for the sake of convenience, we refer to $Q_\eps^{(1)}-Q$ (resp. $Q_\eps^{(2)}-Q$) as the radial direction (resp. tangential direction). Note that these five directions $Q_\eps^{(i)}-Q$, $1\leq i\leq 5$
are orthogonal to one another in the sense that their inner product
$$
  \big(Q_\eps^{(i)}-Q\big): \big(Q_\eps^{(j)}-Q\big)=0, \qquad 1\leq i\neq j\leq 5.
$$
\noindent\textbf{Step 1: radial component.} We first calculate
\begin{equation}
\nabla_{rad}f(Q)\defeq\displaystyle\lim_{\eps\rightarrow 0}\frac{f(Q_\eps^{(1)})-f(Q)}{|Q_\eps^{(1)}-Q|}=\sqrt{\frac23}\frac{\mathrm{d}{f}(Q_\eps^{(1)})}{\mathrm{d}\eps}\bigg|_{\eps=0}.\label{radial-def}
\end{equation}
Let us denote $\rho^{(1)}_\eps$ the associated Boltzmann distribution function of $f(Q_\eps^{(1)})$:
\begin{align*}
\rho^{(1)}_\eps&=\frac{\exp\big\{\mu^{(1)}_1(\eps)x^2+\mu^{(1)}_2(\eps)y^2+\mu^{(1)}_3(\eps)z^2\big\}}{Z^{(1)}_\eps}, \quad(x,y,z)\in\sS, \\
Z^{(1)}_\eps&=\int_{\sS}\exp\big\{\mu^{(1)}_1(\eps)x^2+\mu^{(1)}_2(\eps)y^2+\mu^{(1)}_3(\eps)z^2\big\}\,\ud{S},
\quad\mu^{(1)}_1(\eps)+\mu^{(1)}_2(\eps)+\mu^{(1)}_3(\eps)=0,\\
\frac{1}{Z_\eps^{(1)}}\frac{\partial Z_\eps^{(1)}}{\partial\mu_i^{(1)}}&=\lambda_i^{(1)}(\eps)+\frac13, \qquad 1\leq i\leq 3.
\end{align*}
Here $\lambda_i^{(1)}(\eps)'s$, $1\leq i\leq 3$ are eigenvalues of $Q_{\eps}^{(1)}$.
Clearly, $\rho_0$ (resp. $\mu_i$, i = 1, 2, 3) is the optimal Boltzmann distribution \eqref{Boltzmann-distribution} (resp. Lagrange multipliers)
associated to $Q$. Note that
$$
  \lambda^{(1)}_1(\eps)=\lambda_1-\eps, \quad\lambda^{(1)}_2(\eps)=\lambda_2+\frac{\eps}{2},\quad\lambda^{(1)}_3(\eps)=\lambda_3+\frac{\eps}{2}.
$$
As a consequence, direct computations give
\begin{align*}
\frac{\mathrm{d}f(Q^{(1)}_\eps)}{\mathrm{d}\eps}
&=-\frac{\mathrm{d}\ln Z^{(1)}_\eps}{\mathrm{d}\eps}+\frac{\mathrm{d}}{\mathrm{d}\eps}\sum_{i=1}^3\mu_i^{(1)}(\eps)\Big[\lambda^{(1)}_i(\eps)+\frac13\Big]\\
&=\underbrace{-\frac{\partial\ln Z^{(1)}_\eps}{\partial\mu_i^{(1)}(\eps)}\frac{\mathrm{d}}{\mathrm{d}\eps}\mu_i^{(1)}(\eps)
+\sum_{i=1}^3\frac{\mathrm{d}}{\mathrm{d}\eps}\mu_i^{(1)}(\eps)\Big[\lambda^{(1)}_i(\eps)+\frac13\Big]}_{=0}
+\sum_{i=1}^3\mu_i^{(1)}(\eps)\frac{\mathrm{d}}{\mathrm{d}\eps}\lambda^{(1)}_i(\eps)\\
&=-\mu_1^{(1)}(\eps)+\frac{\mu_2^{(1)}(\eps)+\mu_3^{(1)}(\eps)}{2},
\end{align*}
which gives \eqref{radial-component} after evaluating at $\eps=0$.

\medskip

\noindent\textbf{Step 2: tangential component.}
We proceed to calculate
\begin{equation}
\nabla_{tan}f(Q)\defeq\displaystyle\lim_{\eps\rightarrow 0}\frac{f(Q_\eps^{(2)})-f(Q)}{\big|Q_\eps^{(2)}-Q\big|}
=\sqrt{\frac23}\frac{\mathrm{d}f(Q_\eps^{(2)})}{\mathrm{d}\eps}\bigg|_{\eps=0}.\label{tangential-def}
\end{equation}
Analogously, we denote $\rho^{(2)}_\eps$ the associated Boltzmann distribution function of $f(Q_\eps^{(2)})$:
\begin{align*}
\rho^{(2)}_\eps&=\frac{\exp\big\{\mu^{(2)}_1(\eps)x^2+\mu^{(2)}_2(\eps)y^2+\mu^{(2)}_3(\eps)z^2\big\}}{Z^{(2)}_\eps}, \quad(x,y,z)\in\sS, \\
Z^{(2)}_\eps&=\int_{\sS}\exp\big\{\mu^{(2)}_1(\eps)x^2+\mu^{(2)}_2(\eps)y^2+\mu^{(2)}_3(\eps)z^2\big\}\,\ud{S}, \quad\mu^{(2)}_1(\eps)+\mu^{(2)}_2(\eps)+\mu^{(2)}_3(\eps)=0,\\
\frac{1}{Z_\eps^{(2)}}\frac{\partial Z_\eps^{(2)}}{\partial\mu_i^{(2)}}&=\lambda_i^{(2)}(\eps)+\frac13, \qquad 1\leq i\leq 3.
\end{align*}
Here $\lambda_i^{(2)}(\eps)'s$, $1\leq i\leq 3$ are eigenvalues of $Q_{\eps}^{(2)}$.
Apparently $\rho_0$ (resp. $\mu_i$, i = 1, 2, 3) is the optimal Boltzmann distribution \eqref{Boltzmann-distribution} (resp. Lagrange multipliers)
associated to $Q$. Note that
$$
  \lambda^{(2)}_1(\eps)=\lambda_1, \quad\lambda^{(2)}_2(\eps)=\lambda_2+\eps,\quad\lambda^{(2)}_3(\eps)=\lambda_3-\eps.
$$
Then direct computations give
\begin{align*}
\frac{\mathrm{d}f(Q^{(2)}_\eps)}{\mathrm{d}\eps}
&=\underbrace{-\frac{\partial\ln Z^{(2)}_\eps}{\partial\mu_i^{(2)}(\eps)}\frac{\mathrm{d}}{\mathrm{d}\eps}\mu_i^{(2)}(\eps)
+\sum_{i=1}^3\frac{\mathrm{d}}{\mathrm{d}\eps}\mu_i^{(2)}(\eps)\Big[\lambda^{(2)}_i(\eps)+\frac13\Big]}_{=0}
+\sum_{i=1}^3\mu_i^{(2)}(\eps)\frac{\mathrm{d}}{\mathrm{d}\eps}\lambda_i^{(2)}(\eps)=\mu_2^{(2)}(\eps)-\mu_3^{(2)}(\eps),
\end{align*}
which leads to \eqref{tangential-component} after evaluating at $\eps=0$.

\medskip

\noindent\textbf{Step 3: other components.} We show that all three other components are identically zero. We first calculate
\begin{equation}
\nabla_{3}f(Q)\defeq\displaystyle\lim_{\eps\rightarrow 0}\frac{f(Q_\eps^{(3)})-f(Q)}{\big|Q_\eps^{(3)}-Q\big|}
=\frac{\sqrt{2}}{2}\frac{\mathrm{d}f(Q_\eps^{(3)})}{\mathrm{d}\eps}\bigg|_{\eps=0}.
\end{equation}
Analogously, we denote $\rho^{(3)}_\eps$ the associated Boltzmann distribution function of $f(Q_\eps^{(3)})$:
\begin{align*}
\rho^{(3)}_\eps&=\frac{\exp\big\{\mu^{(3)}_1(\eps)x^2+\mu^{(3)}_2(\eps)y^2+\mu^{(3)}_3(\eps)z^2\big\}}{Z^{(3)}_\eps}, \quad(x,y,z)\in\sS, \\
Z^{(3)}_\eps&=\int_{\sS}\exp\big\{\mu^{(3)}_1(\eps)x^2+\mu^{(3)}_2(\eps)y^2+\mu^{(3)}_3(\eps)z^2\big\}\,\ud{S}, \quad\mu^{(3)}_1(\eps)+\mu^{(3)}_2(\eps)+\mu^{(3)}_3(\eps)=0,\\
\frac{1}{Z_\eps^{(3)}}\frac{\partial Z_\eps^{(3)}}{\partial\mu_i^{(3)}}&=\lambda_i^{(3)}(\eps)+\frac13, \qquad 1\leq i\leq 3.
\end{align*}
Here $\lambda_i^{(3)}(\eps)'s$, $1\leq i\leq 3$ are eigenvalues of $Q_{\eps}^{(3)}$. Once again, $\rho_0$ (resp. $\mu_i$, i = 1, 2, 3) is the optimal Boltzmann distribution \eqref{Boltzmann-distribution} (resp. Lagrange multipliers)
associated to $Q$. Note that
$$
  \lambda^{(3)}_1(\eps)=\frac{\lambda_1+\lambda_2-\sqrt{(\lambda_1-\lambda_2)^2+\eps^2}}{2}, \quad
  \lambda^{(3)}_2(\eps)=\frac{\lambda_1+\lambda_2+\sqrt{(\lambda_1-\lambda_2)^2+\eps^2}}{2},\quad\lambda^{(3)}_3(\eps)=\lambda_3.
$$
Then direct computations give
\begin{align*}
\frac{\mathrm{d}f(Q^{(3)}_\eps)}{\mathrm{d}\eps}
&=\underbrace{-\frac{\partial\ln Z^{(3)}_\eps}{\partial\mu_i^{(3)}(\eps)}\frac{\mathrm{d}}{\mathrm{d}\eps}\mu_i^{(3)}(\eps)
+\sum_{i=1}^3\frac{\mathrm{d}}{\mathrm{d}\eps}\mu_i^{(3)}(\eps)\Big[\lambda^{(3)}_i(\eps)+\frac13\Big]}_{=0}
+\sum_{i=1}^3\mu_i^{(3)}(\eps)\frac{\mathrm{d}}{\mathrm{d}\eps}\lambda_i^{(3)}(\eps) \\
&=\mu_1^{(3)}(\eps)\frac{\mathrm{d}}{\mathrm{d}\eps}\lambda_1^{(3)}(\eps)+\mu_2^{(3)}(\eps)\frac{\mathrm{d}}{\mathrm{d}\eps}\lambda_2^{(3)}(\eps).
%&=-\frac{\mu_1^{(3)}(\eps)}{2}\frac{\eps}{\sqrt{(\lambda_1-\lambda_2)^2+\eps^2}}+\frac{\mu_2^{(3)}(\eps)}{2}\frac{\eps}{\sqrt{(\lambda_1-\lambda_2)^2+\eps^2}}
\end{align*}
\noindent{Case 1:} $\lambda_1\neq\lambda_2$. Then
$$
  \frac{\mathrm{d}f(Q^{(3)}_\eps)}{\mathrm{d}\eps}\bigg|_{\eps=0}=-\frac{\mu_1^{(3)}(\eps)}{2}\frac{\eps}{\sqrt{(\lambda_1-\lambda_2)^2+\eps^2}}\bigg|_{\eps=0}
  +\frac{\mu_2^{(3)}(\eps)}{2}\frac{\eps}{\sqrt{(\lambda_1-\lambda_2)^2+\eps^2}}\bigg|_{\eps=0}=0
$$
\noindent{Case 2:} $\lambda_1=\lambda_2$. Then $\lambda^{(3)}_1(\eps)=\lambda_1-\eps$, $\lambda^{(3)}_2(\eps)=\lambda_1+\eps$, and correspondingly
$$
  \frac{\mathrm{d}f(Q^{(3)}_\eps)}{\mathrm{d}\eps}\bigg|_{\eps=0}=-\mu_1+\mu_2=0,
$$
due to {\color{black}\eqref{lemma-orders-mu}} and the assumption $\lambda_1=\lambda_2$. In a similar way we can show that the other two components are identically zero.
Hence $\nabla{f}(Q)-1/3\tr(\nabla{f})\mathbb{I}_3$ depends on $\nabla_{rad}f(Q)$ and $\nabla_{tan}f(Q)$ only.
\end{proof}

As an immediate consequence, we have

\begin{corollary}\label{corollary-ratio}
For any physical $Q$-tensor of the form \eqref{Q-diagonal}, it holds
\begin{equation}
1\leq\frac{\big|\nabla_{\Q}f(Q)\big|}{|\nabla_{rad}f(Q)|}\leq 2.
\end{equation}
\end{corollary}
\begin{proof}
It suffices to prove the upper bound. It follows directly from Lemma \ref{lemma-rate-gradient} that
$$
  \big|\nabla_{\Q}f(Q)\big|^2=\big|\nabla_{rad}f(Q)\big|^2+\big|\nabla_{tan}f(Q)\big|^2=\frac32\mu_1^2+\frac12(\mu_1+2\mu_2)^2.
$$
By \eqref{Boltzmann-distribution} and Lemma \ref{lemma-orders-mu}, it is easy to check
$$
  3\mu_1\leq\mu_1+2\mu_2\leq 0,
$$
which implies
$$
  \big|\nabla_{\Q}f(Q)\big|^2\leq\frac32\mu_1^2+\frac12(3\mu_1)^2=6\mu_1^2,
$$
and further together with \eqref{radial-component} yields
$$
  1\leq\frac{\big|\nabla_{\Q}f(Q)\big|^2}{|\nabla_{rad}f(Q)|^2}\leq\frac{6\mu_1^2}{\frac{3}{2}\mu_1^2}=4.
$$
\end{proof}

\begin{remark}
It follows from Lemma \ref{lemma-rate-gradient} and Corollary \ref{corollary-ratio} that, to estimate the blowup rate of $|\nabla_{\Q}f(Q)|$ as $\lambda_1(Q)\rightarrow-1/3$, it suffices to estimate $\mu_1$.
\end{remark}

\noindent[Proof of Theorem \ref{theorem-blowup-rate}]
The proof of Theorem \ref{theorem-blowup-rate} consists of the following two propositions, which provides the upper bound and lower bound
of $|\nabla_{\Q}f|$, respectively.

\begin{proposition}\label{propotion-upper-bound}
For any physical $Q$-tensor of the form \eqref{Q-diagonal}, there exists a small computable constant $\eps_0>0$, such that
\begin{equation}
\big|\nabla_{\Q}f(Q)\big|\geq\frac{C_1}{\lambda_1+\frac13}, \qquad\mbox{provided }\;0<\lambda_1+\frac13<\eps_0,
\end{equation}
where $C_1$ is given in \eqref{constant-C1}.
\end{proposition}
\begin{proof}
To begin with, we see from Lemma \ref{lemma-rate-gradient} that
\begin{equation}
\big|\nabla_{\Q}f(Q)\big|^2\geq\big|\nabla_{rad}f(Q)\big|^2=\frac32\mu_1^2,
\end{equation}
where $\mu_1$ is the Lagrange multiplier associated with the Boltzmann distribution function of $f(Q)$ given in \eqref{Boltzmann-distribution}.
Hence it remains to estimate $\mu_1$ in terms of $\lambda_1(Q)+1/3$, as $\lambda_1(Q)$ approaches $-1/3$. From \eqref{second-moment-new-1},
\eqref{second-moment-new-2} we see
\begin{align}\label{second-moment}
\lambda_1(Q)+\frac13=\frac{\int_{\sS}x^2\exp(-\nu_1x^2-\nu_2y^2)\,\ud{S}}{\int_{\sS}\exp(-\nu_1x^2-\nu_2y^2)\,\ud{S}}.
\end{align}
Here $\nu_1, \nu_2$ are given in \eqref{def-nu-1}. Recall that
\begin{equation}\label{large-A}
\nu_1 \gg 1, \quad \nu_2\geq 0.
\end{equation}

\medskip

\noindent\textbf{Step 1: Estimating the numerator in \eqref{second-moment}.}

\medskip

Using the coordinate system
\begin{equation}\label{coordinate-system}
v(x,\theta)\defeq (x, \sqrt{1-x^2}\cos\theta, \sqrt{1-x^2}\sin\theta), \quad 0\leq x\leq 1,\;0\leq\theta\leq 2\pi,
\end{equation}
whose surface element is given by $|\frac{\partial{v}}{\partial{x}}\times\frac{\partial{v}}{\partial\theta}|=1$, we get
\begin{align}\label{numerator-1}
\int_{\sS}x^2\exp(-\nu_1x^2-\nu_2y^2)\,\ud{S}&=2\int_{\sS\cap\{x\geq 0\}}x^2\exp(-\nu_1x^2-\nu_2y^2)\,\ud{S}\non\\
&=2\int_0^1x^2e^{-\nu_1x^2}\Big[\int_0^{2\pi}e^{-\nu_2(1-x^2)\cos^2\theta}\,\ud\theta\Big]\,\ud{x}.
\end{align}
Note that the zeroth modified Bessel function of first kind \cite{NIST10} is represented by
\begin{equation}
\IO(\xi)=\frac{1}{\pi}\int_0^{\pi}\exp(\xi\cos\theta)\,\ud\theta=\displaystyle\sum_{m=0}^{+\infty}\frac{1}{(m!)^2}\big(\frac{\xi}{2}\big)^{2m},
\end{equation}
and $\IO(\xi)=\IO(-\xi)$, hence we have
\begin{align*}
&\int_0^{2\pi}e^{-\nu_2(1-x^2)\cos^2\theta}\,\ud\theta=\exp\Big\{\frac{-\nu_2(1-x^2)}{2}\Big\}\int_0^{2\pi}\exp\Big\{\frac{-\nu_2(1-x^2)\cos(2\theta)}{2}\Big\}\,\ud\theta\\
&\overset{(\eta=2\theta)}{=}\exp\Big\{\frac{-\nu_2(1-x^2)}{2}\Big\}\int_0^{2\pi}\exp\Big\{\frac{-\nu_2(1-x^2)\cos(\eta)}{2}\Big\}\,\ud\eta\\
&=\exp\Big\{\frac{-\nu_2(1-x^2)}{2}\Big\}\Big[\int_0^{\pi}\exp\Big\{\frac{-\nu_2(1-x^2)\cos(\eta)}{2}\Big\}\,\ud\eta+\int_0^{\pi}\exp\Big\{\frac{\nu_2(1-x^2)\cos(\eta)}{2}\Big\}\,\ud\eta\Big]\\
&=2\pi\exp\Big\{\frac{-\nu_2(1-x^2)}{2}\Big\}\IO\Big[\frac{\nu_2(1-x^2)}{2}\Big].
\end{align*}
Inserting the above identity into \eqref{numerator-1}, together with \eqref{large-A}, we obtain
\begin{align}\label{numerator-2}
\int_{\sS}x^2\exp(-\nu_1x^2-\nu_2y^2)\,\ud{S}
&=4\pi\int_0^1x^2e^{-\nu_1x^2}\exp\Big\{\frac{-\nu_2(1-x^2)}{2}\Big\}\IO\Big[\frac{\nu_2(1-x^2)}{2}\Big]\,\ud{x}\non\\
&\geq 4\pi\int_0^{\frac{1}{\sqrt{\nu_1}}}x^2e^{-\nu_1x^2}\exp\Big\{\frac{-\nu_2(1-x^2)}{2}\Big\}\IO\Big[\frac{\nu_2(1-x^2)}{2}\Big]\,\ud{x}\non\\
&\geq 4\pi{e}^{-1}\int_0^{\frac{1}{\sqrt{\nu_1}}}x^2\exp\Big\{\frac{-\nu_2(1-x^2)}{2}\Big\}\IO\Big[\frac{\nu_2(1-x^2)}{2}\Big]\,\ud{x} {\color{black}.}
\end{align}
Meanwhile, since \cite{NIST10}
$$
  \frac{\mathrm{d}}{\mathrm{d}\xi}\big[e^{-\xi}\IO(\xi)\big]=e^{-\xi}\big[\mathrm{I}_1(\xi)-\IO(\xi)\big]
  =\frac{e^{-\xi}}{\pi}\int_0^{\pi}(\cos\theta-1)\exp(\xi\cos\theta)\,\ud\theta<0,
$$
the function $\xi\mapsto e^{-\xi}\IO(\xi)$ is strictly decreasing. Correspondingly we have
\begin{equation}\label{monotone-function}
x\mapsto \exp\Big\{\frac{-\nu_2(1-x^2)}{2}\Big\}\IO\Big[\frac{\nu_2(1-x^2)}{2}\Big] \quad\mbox{is strictly increasing for }\;x\in [0, 1].
\end{equation}
By virtue of \eqref{monotone-function}, we get
$$
  \inf_{x\in[0, 1/\sqrt{\nu_1}]}\exp\Big\{\frac{-\nu_2(1-x^2)}{2}\Big\}\IO\Big[\frac{\nu_2(1-x^2)}{2}\Big]=\exp\Big(\frac{-\nu_2}{2}\Big)\IO\Big(\frac{\nu_2}{2}\Big),
$$
which together with \eqref{numerator-2} implies
\begin{align}\label{numerator-estimate-1}
\int_{\sS}x^2\exp(-\nu_1x^2-\nu_2y^2)\,\ud{S}\geq\frac{4\pi}{e}\exp\Big(\frac{-\nu_2}{2}\Big)\IO\Big(\frac{\nu_2}{2}\Big)\int_0^{\frac{1}{\sqrt{\nu_1}}}x^2\,\ud{x}
=\frac{4\pi}{3e}\exp\Big(\frac{-\nu_2}{2}\Big)\IO\Big(\frac{\nu_2}{2}\Big)\nu_1^{-\frac32}.
\end{align}

\noindent\textbf{Step 2: Estimating the denominator in \eqref{second-moment}.}

\medskip

Similar to the last step, we have
\begin{align}
&\int_{\sS}\exp(-\nu_1x^2-\nu_2y^2)\,\ud{S}
=4\pi\int_0^1e^{-\nu_1x^2}\exp\Big\{\frac{-\nu_2(1-x^2)}{2}\Big\}\IO\Big[\frac{\nu_2(1-x^2)}{2}\Big]\,\ud{x}\non\\
&=4\pi\int_0^{\frac{1}{\sqrt{2}}}e^{-\nu_1x^2}\exp\Big\{\frac{-\nu_2(1-x^2)}{2}\Big\}\IO\Big[\frac{\nu_2(1-x^2)}{2}\Big]\,\ud{x}\non\\
&\qquad+4\pi\int_{\frac{1}{\sqrt{2}}}^1e^{-\nu_1x^2}\exp\Big\{\frac{-\nu_2(1-x^2)}{2}\Big\}\IO\Big[\frac{\nu_2(1-x^2)}{2}\Big]\,\ud{x}.
\end{align}
Recall \eqref{monotone-function}, then we see
\begin{align*}
&\int_{\frac{1}{\sqrt{2}}}^1e^{-\nu_1x^2}\exp\Big\{\frac{-\nu_2(1-x^2)}{2}\Big\}\IO\Big[\frac{\nu_2(1-x^2)}{2}\Big]\,\ud{x}
\leq\int_{\frac{1}{\sqrt{2}}}^1e^{-\frac{\nu_1}{2}}\,\ud{x}=\frac{2-\sqrt{2}}{2}e^{-\frac{\nu_1}{2}},\\
&\int_0^{\frac{1}{\sqrt{2}}}e^{-\nu_1x^2}\exp\Big\{\frac{-\nu_2(1-x^2)}{2}\Big\}\IO\Big[\frac{\nu_2(1-x^2)}{2}\Big]\,\ud{x}
\leq\exp\Big(\frac{-\nu_2}{4}\Big)\IO\Big(\frac{\nu_2}{4}\Big)\int_0^{\frac{1}{\sqrt{2}}}e^{-\nu_1x^2}\,\ud{x} \\
&\leq\exp\Big(\frac{-\nu_2}{4}\Big)\IO\Big(\frac{\nu_2}{4}\Big)\int_0^{\infty}e^{-\nu_1x^2}\,\ud{x}=\exp\Big(\frac{-\nu_2}{4}\Big)\IO\Big(\frac{\nu_2}{4}\Big)\sqrt{\frac{\pi}{4\nu_1}} {\color{black}.}
\end{align*}
Since $\xi\mapsto e^{-\xi/4}\IO(\xi/4)$ is decreasing, while $\xi\mapsto \IO(\xi/4)$ is increasing for $\xi\geq 0$, and $\nu_2\leq \nu_1$, $\nu_1\gg 1$,
there exists a computable, universal constant $A_0$ such that
\begin{align*}
\frac{2-\sqrt{2}}{2}e^{-\frac{\nu_1}{2}}\leq\exp\Big(\frac{-\nu_1}{4}\Big)\sqrt{\frac{\pi}{4\nu_1}}\leq\exp\Big(\frac{-\nu_1}{4}\Big)\IO\Big(\frac{\nu_1}{4}\Big)\sqrt{\frac{\pi}{4\nu_1}}
\leq\exp\Big(\frac{-\nu_2}{4}\Big)\IO\Big(\frac{\nu_2}{4}\Big)\sqrt{\frac{\pi}{4\nu_1}}, \qquad\forall \nu_1\geq A_0.
\end{align*}
To sum up, we conclude
\begin{equation}\label{denominator-estimate-1}
\int_{\sS}\exp(-\nu_1x^2-\nu_2y^2)\,\ud{S}\leq 4\pi\exp\Big(\frac{-\nu_2}{4}\Big)\IO\Big(\frac{\nu_2}{4}\Big)\sqrt{\frac{\pi}{\nu_1}}, \qquad\forall \nu_1\geq A_0.
\end{equation}

\noindent\textbf{Step 3: Combining both estimates in \eqref{second-moment}.}

\medskip

We get immediately from \eqref{numerator-estimate-1} and \eqref{denominator-estimate-1} that
\begin{equation}\label{second-moment-bound}
\lambda_1(Q)+\frac13=\frac{1}{3e\sqrt{\pi}\nu_1}\frac{\exp\big(\frac{-\nu_2}{2}\big)\IO\big(\frac{\nu_2}{2}\big)}{\exp\big(\frac{-\nu_2}{4}\big)\IO\big(\frac{\nu_2}{4}\big)}, \qquad\forall \nu_1\geq A_0.
\end{equation}
Since $-3\mu_1\geq \nu_1\geq-\frac32\mu_1$, it remains to bound the last fraction in \eqref{second-moment-bound}. Since $e^{-\xi}\IO(\xi)>0$, $\forall\xi\geq 0$, and is equal to $1$ at $\xi=0$, it suffices to show
\begin{equation}\label{limit-inf}
\liminf_{\xi\rightarrow+\infty}\frac{e^{-\xi}\IO(\xi)}{e^{\frac{-\xi}{2}}\IO(\frac{\xi}{2})}>0 {\color{black}.}
\end{equation}
The  {\color{black}asymptotic} expansion of $e^{-\xi}\IO(\xi)$, for $\xi\gg 1$, is \cite{NIST10}
$$
  e^{-\xi}\IO(\xi)=\frac{1}{\sqrt{2\pi}}\Big[\xi^{-\frac12}+\frac18\xi^{-\frac32}+\frac{9}{128}\xi^{-\frac52}+O(\xi^{-\frac72}) \Big],
$$
hence the limit in \eqref{limit-inf} is equal to $1/\sqrt{2}$. Thus one can establish from \eqref{second-moment-bound} and the fact $-3\mu_1\geq \nu_1\geq-\frac32\mu_1$ that
\begin{equation*}
\lambda_1(Q)+\frac13\geq\frac{1}{\nu_1}\frac{1}{3e\sqrt{\pi}}\inf_{\xi\geq 0}\frac{e^{-\xi}\IO(\xi)}{e^{\frac{-\xi}{2}}\IO(\frac{\xi}{2})}, \qquad\forall \nu_1\geq A_0,
\end{equation*}
which further gives
\begin{equation}\label{inequality-last}
\lambda_1(Q)+\frac13\geq-\frac{1}{\mu_1}\frac{1}{9e\sqrt{\pi}}\inf_{\xi\geq 0}\frac{e^{-\xi}\IO(\xi)}{e^{\frac{-\xi}{2}}\IO(\frac{\xi}{2})},
\quad\mbox{provided }\; \mu_1< -\frac{2A_0}{3}.
\end{equation}
In view of Remark \ref{remark-small-eps}, there exists a small computable constant $\eps_0>0$, such that
$$
  \mu_1< -\frac{2A_0}{3}, \quad\mbox{provided }\; 0<\lambda_1(Q)+\frac13<\eps_0.
$$
In all, using \eqref{constant-C1} we conclude that from \eqref{inequality-last} that
\begin{equation}
\big|\nabla_{\Q}f(Q)\big|\geq\big|\nabla_{rad}f(Q)\big|\geq\frac{\sqrt{6}}{2}\mu_1\geq\frac{C_1}{\lambda_1(Q)+\frac13}, \quad\mbox{provided }\; 0<\lambda_1(Q)+\frac13<\eps_0,
\end{equation}
completing the proof.
\end{proof}

\begin{proposition}\label{propotion-lower-bound}
For any physical $Q$-tensor of the form \eqref{Q-diagonal}, there exists a small computable constant $\eps_0>0$, such that
\begin{equation}
\big|\nabla_{\Q}f(Q)\big|\leq\frac{C_2}{\lambda_1+\frac13}, \qquad\mbox{provided }\;0<\lambda_1+\frac13<\eps_0,
\end{equation}
where $C_2$ is given in \eqref{constant-C1}.
\end{proposition}
\begin{proof}
In contrast to the proof of Proposition \ref{propotion-upper-bound}, we need to obtain suitable upper bound on the numerator of \eqref{second-moment}, but lower
bound on the denominator of \eqref{second-moment}.

\medskip

\noindent\textbf{Step 1: Estimating the numerator in \eqref{second-moment}.}

\medskip

Using the coordinate system \eqref{coordinate-system}, similar to the proof of Proposition \ref{propotion-upper-bound} one can establish
\begin{align}\label{numerator-3}
\int_{\sS}x^2\exp(-\nu_1x^2-\nu_2y^2)\,\ud{S}&=4\pi\int_0^1x^2e^{-\nu_1x^2}\exp\Big\{\frac{-\nu_2(1-x^2)}{2}\Big\}\IO\Big[\frac{\nu_2(1-x^2)}{2}\Big]\,\ud{x}\non\\
&=4\pi\int_0^{\frac{1}{\sqrt{2}}}x^2e^{-\nu_1x^2}\exp\Big\{\frac{-\nu_2(1-x^2)}{2}\Big\}\IO\Big[\frac{\nu_2(1-x^2)}{2}\Big]\,\ud{x}\non\\
&\qquad+4\pi\int_{\frac{1}{\sqrt{2}}}^1x^2e^{-\nu_1x^2}\exp\Big\{\frac{-\nu_2(1-x^2)}{2}\Big\}\IO\Big[\frac{\nu_2(1-x^2)}{2}\Big]\,\ud{x}\non\\
&\doteq 4\pi(J_1+J_2).
\end{align}
In view of \eqref{monotone-function}, we get
\begin{equation}
J_2\leq\int_{\frac{1}{\sqrt{2}}}^1x^2e^{-\nu_1x^2}\,\ud{x}=\frac{2-\sqrt{2}}{2}e^{-\frac{\nu_1}{2}},
\end{equation}
and
\begin{align}
J_1&\leq\int_0^{\frac{1}{\sqrt{2}}}x^2e^{-\nu_1x^2}\exp\Big(-\frac{\nu_2}{4}\Big)\IO\Big(\frac{\nu_2}{4}\Big)\,\ud{x}
=\exp\Big(-\frac{\nu_2}{4}\Big)\IO\Big(\frac{\nu_2}{4}\Big)\left[-\frac{e^{-\frac{\nu_1}{2}}}{2\sqrt{2}\nu_1}+\frac{1}{2\nu_1}\int_0^{\frac{1}{\sqrt{2}}}e^{-\nu_1x^2}\,\ud{x}\right]\non\\
&\leq\exp\Big(-\frac{\nu_2}{4}\Big)\IO\Big(\frac{\nu_2}{4}\Big)\Big[-\frac{e^{-\frac{\nu_1}{2}}}{2\sqrt{2}\nu_1}+\frac{1}{2\nu_1}\int_0^{+\infty}e^{-\nu_1x^2}\,\ud{x}\Big]\non\\
&=\exp\Big(-\frac{\nu_2}{4}\Big)\IO\Big(\frac{\nu_2}{4}\Big)\Big[-\frac{e^{-\frac{\nu_1}{2}}}{2\sqrt{2}\nu_1}+\frac{\sqrt{\pi}}{4}\nu_1^{-\frac32}\Big].
\end{align}
Since $\nu_2\leq \nu_1$, and $\xi\mapsto e^{-\xi}\IO(\xi)$ is decreasing, we have
$$
  0<\frac{e^{-\frac{\nu_1}{2}}}{\exp\big(-\frac{\nu_2}{4}\big)\IO\big(\frac{\nu_2}{4}\big)\sqrt{\pi}\big/4\nu_1^{\frac32}}
  \leq\frac{e^{-\frac{\nu_1}{2}}}{\exp\big(-\frac{\nu_1}{4}\big)\IO\big(\frac{\nu_1}{4}\big)\sqrt{\pi}\big/4\nu_1^{\frac32}}
  =\frac{4\nu_1^{\frac32}}{e^{\frac{\nu_1}{4}}\IO\big(\frac{\nu_1}{4}\big)\sqrt{\pi}}
  \longrightarrow 0, \quad\mbox{as }\;\nu_1\rightarrow +\infty.
$$
Hence there exists a computable constant $A_0>0$, such that
\begin{align*}
J_1+J_2&\leq\frac{2-\sqrt{2}}{2}e^{-\frac{\nu_1}{2}}+\exp\Big(-\frac{\nu_2}{4}\Big)\IO\Big(\frac{\nu_2}{4})\Big[-\frac{e^{-\frac{\nu_1}{2}}}{2\sqrt{2}\nu_1}
+\frac{\sqrt{\pi}}{4}\nu_1^{-\frac32}\Big]\\
&\leq 4\pi\exp\Big(-\frac{\nu_2}{4}\Big)\IO\Big(\frac{\nu_2}{4}\Big)\sqrt{\pi}\nu_1^{-\frac32}, \qquad\forall \nu_1>A_0,
\end{align*}
which inserts into \eqref{numerator-3} yields
\begin{equation}\label{numerator-estimate-2}
\int_{\sS}x^2\exp(-\nu_1x^2-\nu_2y^2)\,\ud{S}\leq 4\pi(J_1+J_2)\leq 4\pi\leq \exp\Big(-\frac{\nu_2}{4}\Big)\IO\Big(\frac{\nu_2}{4}\Big)\sqrt{\pi}\nu_1^{-\frac32}, \quad\forall \nu_1>A_0.
\end{equation}

\noindent\textbf{Step 2: Estimating the denominator in \eqref{second-moment}.}

\medskip

Using \eqref{large-A}, the coordinate system \eqref{coordinate-system} and \eqref{monotone-function}, we get
\begin{align}\label{denominator-estimate-2}
\int_{\sS}\exp(-\nu_1x^2-\nu_2y^2)\,\ud{S}&=4\pi\int_0^1e^{-\nu_1x^2}\exp\Big\{\frac{-\nu_2(1-x^2)}{2}\Big\}\IO\Big[\frac{\nu_2(1-x^2)}{2}\Big]\,\ud{x}\non\\
&\geq 4\pi\int_0^{\frac{1}{\sqrt{\nu_1}}}e^{-\nu_1x^2}\exp\Big\{\frac{-\nu_2(1-x^2)}{2}\Big\}\IO\Big[\frac{\nu_2(1-x^2)}{2}\Big]\,\ud{x}\non\\
&\geq \frac{4\pi}{e\sqrt{\nu_1}}\exp\Big[-\frac{\nu_2}{2}\Big]\IO\Big(\frac{\nu_2}{2}\Big) {\color{black}.}
\end{align}

\noindent\textbf{Step 3: Completing the proof.}

\medskip

Combining \eqref{numerator-estimate-2} and \eqref{denominator-estimate-2}, we see that
\begin{align}
\lambda_1(Q)+\frac13\leq\frac{\exp\big(-\frac{\nu_2}{4}\big)\IO\big(\frac{\nu_2}{4}\big)\sqrt{\pi}\nu_1^{-\frac32}}{\frac{1}{e\sqrt{\nu_1}}\exp\big(-\frac{\nu_2}{2}\big)\IO\big(\frac{\nu_2}{2}\big)}
\leq\frac{e\sqrt{\pi}}{\nu_1}\sup_{\xi\geq 0}\frac{\exp\big(-\frac{\xi}{4}\big)\IO\big(\frac{\xi}{4}\big)}{\exp\big(-\frac{\xi}{2}\big)\IO\big(\frac{\xi}{2}\big)}, \qquad\forall \nu_1>A_0
\end{align}
where
$$
  \sup_{\xi\geq 0}\frac{\exp\big(-\frac{\xi}{4}\big)\IO\big(\frac{\xi}{4}\big)}{\exp\big(-\frac{\xi}{2}\big)\IO\big(\frac{\xi}{2}\big)}
\geq\frac{\IO(0)}{\IO(0)}=1.
$$
Therefore, together with the fact that $\mu_1\leq\mu_2\leq -\mu_1/2$, we know
$$
  -\mu_1\leq-2\mu_1-\mu_2=\nu_1\leq\frac{1}{\lambda_1(Q)+\frac13}e\sqrt{\pi}\sup_{\xi\geq 0}\frac{\exp\big(-\frac{\xi}{4}\big)\IO\big(\frac{\xi}{4}\big)}{\exp\big(-\frac{\xi}{2}\big)\IO\big(\frac{\xi}{2}\big)}, \qquad\forall \nu_1>A_0.
$$
Then following the same argument in the proof of Proposition \ref{propotion-lower-bound}, we conclude from
Corollary \ref{corollary-ratio} that as $\lambda_1(Q)\rightarrow -1/3$, it holds
\begin{equation}
\big|\nabla_{\Q}f(Q)\big|\leq 2|\nabla_{rad}f(Q)|=-\sqrt{6}\mu_1=\frac{C_2}{\lambda_1(Q)+\frac13},
\end{equation}
where $C_2$ is given in Theorem \ref{constant-C1}, completing the proof.

\end{proof}

\begin{remark}
Using similar argument, it is expected that the estimate for second order derivatives of $f$ near its physical boundary could be achieved.
\end{remark}

\section{Acknowledgements}

We would like to thank Professor John Ball for his kind discussions, especially pointing out several useful references to us.
X.~Y. Lu's work is supported by his
NSERC Discovery Grant ``Regularity of minimizers and pattern formation in geometric minimization problems''. X. Xu's work is supported by the NSF grant DMS-2007157 and the Simons Foundation Grant No. 635288.
W.~J. Zhang's work is supported by NSF Grant DMS-1818861.

%%%%%%%%%%%%%%%%%%%%%%%%%%%%%%%%%%%%%%%%%%%%%%%%%%%%%%%%%%%%%%%%%%%%%%%%%%%%%%%%%%%%%%%%%%%%%%%%%%%%%%%%%%%%%%%%%%%%%%%%%%

\end{document}